\documentclass[12pt]{amsart}
\usepackage{amsmath}
\usepackage{amsfonts}
\usepackage{amssymb}
\usepackage{bbding}
\usepackage{stmaryrd}
\usepackage{txfonts}
\usepackage{graphicx}
\usepackage{epsfig}
\usepackage{xypic}
\usepackage{tikz}
\usepackage{longtable}
\usepackage[all]{xy}
\usepackage{pgflibraryarrows}
\usepackage{pgflibrarysnakes}
\usepackage[shortlabels]{enumitem}
\usepackage{ifpdf}
\ifpdf
  \usepackage[colorlinks,final,backref=page,hyperindex]{hyperref}
\else
  \usepackage[colorlinks,final,backref=page,hyperindex,hypertex]{hyperref}
\fi
\usepackage{xspace}

\newtheorem{theorem}{Theorem}[section]
\newtheorem{lemma}[theorem]{Lemma}
\newtheorem{coro}[theorem]{Corollary}

\newtheorem{prop}[theorem]{Proposition}
\theoremstyle{definition}
\newtheorem{defn}[theorem]{Definition}

\newtheorem{exam}[theorem]{Example}

\newcommand{\nc}{\newcommand}
\newcommand{\delete}[1]{}

\nc{\tred}[1]{\textcolor{red}{#1}}
\nc{\tblue}[1]{\textcolor{blue}{#1}} \nc{\tgreen}[1]{\textcolor{green}{#1}} \nc{\tpurple}[1]{\textcolor{purple}{#1}} \nc{\btred}[1]{\textcolor{red}{\bf #1}} \nc{\btblue}[1]{\textcolor{blue}{\bf #1}} \nc{\btgreen}[1]{\textcolor{green}{\bf #1}} \nc{\btpurple}[1]{\textcolor{purple}{\bf #1}}

\newcommand{\efootnote}[1]{}

\nc{\mlabel}[1]{\label{#1}}  
\nc{\mcite}[2][]{\cite[#1]{#2}}  
\nc{\mref}[1]{\ref{#1}}  
\nc{\mbibitem}[1]{\bibitem{#1}} 

\delete{
\nc{\mlabel}[1]{\label{#1}  
{\hfill \hspace{1cm}{\bf{{\ }\hfill(#1)}}}}
\nc{\mcite}[1]{\cite{#1}}  
\nc{\mref}[1]{\ref{#1}{{\bf{{\ }(#1)}}}}  
\nc{\mbibitem}[1]{\bibitem[\bf #1]{#1}} 
}

\renewcommand\geq{\geqslant}
\renewcommand\leq{\leqslant}

\renewcommand\bar[1]{\overline{#1}}

\nc{\nz}{\varepsilon}
\nc{\Id}{\mathrm{Id}}
\nc{\map}[2]{{#2}^{#1}}
\nc{\gp}{B}

\nc{\Irr}{\mathrm{Irr}}
\nc{\vx}{\sigma} \nc{\vy}{\tau} \nc{\dvx}{\sigma^{(1)}} \nc{\dvy}{\tau^{(1)}} \nc{\done}{\vep} \nc{\mcitep}[1]{\mcite{#1}} \nc{\wt}{\mathrm{wt}} \nc{\bre}[1]{|#1|} \nc{\mapmonoid}{\frakM} \nc{\disjoint}{\frakM'}
\nc{\ncpoly}[1]{\langle #1\rangle}  
\nc{\mapm}[1]{\lfloor\!|{#1}|\!\rfloor}
\nc{\diff}[1]{{}^\NC\{ #1 \}} \nc{\disj}[1]{\{{#1}\}'} \nc{\mdisj}[1]{\frakM'(#1)} \nc{\brho}{\bar{\rho}} \nc{\om}{\bar{\frakm}} \nc{\frakn}{\mathfrak n} \nc{\ddeg}[1]{^{(#1)}} \nc{\opset}{X} \nc{\genset}{{Z}} \nc{\NC}{\mathrm{{NC}}} \nc{\leaf}{\mathrm{leaf}} \nc{\twig}{\mathrm{twig}} \nc{\fe}{\mathrm{fl}} \nc{\munderline}[1]{#1} \nc{\bo}{o} \nc{\dep}{\mathrm{depth}} \nc{\ofe}{\mathrm{ofl}} \nc{\dfe}{\mathrm{dfe}} \nc{\fex}{\mathrm{fex}} \nc{\dl}{\mathrm{dlex}} \nc{\db}{\mathrm{db}} \nc{\lex}{\mathrm{lex}} \nc{\clex}{\mathrm{clex}} \nc{\dgp}{\mathrm{dgp}} \nc{\dgx}{\mathrm{dgx}} \nc{\br}{\mathrm{br}} \nc{\obd}{\mathrm{odb}} \nc{\ob}{\mathrm{ob}}
\nc{\pie}{\mathrm{PIE}}
\nc{\rbo}{\mathrm{RBO}}
\nc{\supp}{\mathcal{S}}
\nc{\nul}{\mathcal{Z}}

\nc{\bin}[2]{ (_{\stackrel{\scs{#1}}{\scs{#2}}})}  
\nc{\binc}[2]{ \left (\!\! \begin{array}{c} \scs{#1}\\
    \scs{#2} \end{array}\!\! \right )}  
\nc{\bincc}[2]{  \left ( {\scs{#1} \atop
    \vspace{-1cm}\scs{#2}} \right )}  
\nc{\bs}{\bar{S}} \nc{\cosum}{\sqsubset} \nc{\la}{\longrightarrow} \nc{\rar}{\rightarrow} \nc{\dar}{\downarrow} \nc{\dprod}{**} \nc{\dap}[1]{\downarrow \rlap{$\scriptstyle{#1}$}} \nc{\md}[1]{\bar{#1}} \nc{\uap}[1]{\uparrow \rlap{$\scriptstyle{#1}$}} \nc{\defeq}{\stackrel{\rm def}{=}} \nc{\disp}[1]{\displaystyle{#1}} \nc{\dotcup}{\ \displaystyle{\bigcup^\bullet}\ } \nc{\gzeta}{\bar{\zeta}} \nc{\hcm}{\ \hat{,}\ } \nc{\hts}{\hat{\otimes}} \nc{\barot}{{\otimes}} \nc{\free}[1]{\bar{#1}} \nc{\uni}[1]{\tilde{#1}} \nc{\hcirc}{\hat{\circ}} \nc{\leng}{\ell} \nc{\lleft}{[} \nc{\lright}{]} \nc{\lc}{\lfloor} \nc{\rc}{\rfloor}
\nc{\lb}{[} 
\nc{\rb}{]} 
\nc{\curlyl}{\left \{ \begin{array}{c} {} \\ {} \end{array}
    \right.  \!\!\!\!\!\!\!}
\nc{\curlyr}{ \!\!\!\!\!\!\!
    \left. \begin{array}{c} {} \\ {} \end{array}
    \right \} }
\nc{\longmid}{\left | \begin{array}{c} {} \\ {} \end{array}
    \right. \!\!\!\!\!\!\!}
\nc{\onetree}{\bullet} \nc{\ora}[1]{\stackrel{#1}{\rar}}
\nc{\ola}[1]{\stackrel{#1}{\la}}
\nc{\ot}{\otimes} \nc{\mot}{{{\boxtimes\,}}} \nc{\otm}{\overline{\boxtimes}} \nc{\sprod}{\bullet} \nc{\scs}[1]{\scriptstyle{#1}} \nc{\mrm}[1]{{\rm #1}} \nc{\msum}{\sum\limits}
\nc{\margin}[1]{\marginpar{\rm #1}}   
\nc{\dirlim}{\displaystyle{\lim_{\longrightarrow}}\,} \nc{\invlim}{\displaystyle{\lim_{\longleftarrow}}\,} \nc{\mvp}{\vspace{0.3cm}} \nc{\tk}{^{(k)}} \nc{\tp}{^\prime} \nc{\ttp}{^{\prime\prime}} \nc{\svp}{\vspace{2cm}} \nc{\vp}{\vspace{8cm}} \nc{\proofbegin}{\noindent{\bf Proof: }}
\nc{\proofend}{$\blacksquare$ \vspace{0.3cm}}
\nc{\modg}[1]{\!<\!\!{#1}\!\!>}
\nc{\intg}[1]{F_C(#1)} \nc{\lmodg}{\!<\!\!} \nc{\rmodg}{\!\!>\!} \nc{\cpi}{\widehat{\Pi}}
\nc{\sha}{{\mbox{\cyr X}}}  
\nc{\shap}{{\mbox{\cyrs X}}} 
\nc{\shpr}{\diamond}    
\nc{\shp}{\ast} \nc{\shplus}{\shpr^+}
\nc{\shprc}{\shpr_c}    
\nc{\msh}{\ast} \nc{\zprod}{m_0} \nc{\oprod}{m_1} \nc{\vep}{\varepsilon} \nc{\labs}{\mid\!} \nc{\rabs}{\!\mid}
\nc{\astarrow}{\overset{\raisebox{-3pt}{$\ast$}}{\rightarrow}}

\nc{\EEsym}{\mathbb{E}sym}
\nc{\sym}{\mrm{Sym}}
\nc{\nsym}{\mrm{NSym}}
\nc{\qsym}{\mrm{QSym}}
\nc{\Ensym}{\mrm{ENSym}}
\nc{\Wcsym}{\mrm{WCSym}}
\nc{\Wcqsym}{\mrm{WCQSym}}
\nc{\syms}{symmetric functions\xspace}
\nc{\eqsym}{extended quasi-symmetric function\xspace}
\nc{\eqsyms}{extended quasi-symmetric functions\xspace}
\nc{\Eqsyms}{Extended Quasi-symmetric functions\xspace}
\nc{\Esyms}{Extended symmetric functions\xspace}
\nc{\sgqsym}{quasi-symmetric function with semigroup exponents\xspace}
\nc{\sgqsyms}{quasi-symmetric functions with semigroup exponents\xspace}
\nc{\Sgqsyms}{Quasi-symmetric functions with semigroup exponents\xspace}
\nc{\SGQSYM}{\mrm{SGQSYM}}
\nc{\emzv}{extended multiple zeta value}
\nc{\emzvs}{extended multiple zeta values}
\nc{\sgfps}{formal power series with semigroup exponent\xspace}
\nc{\nsymg}{\mathrm{NSym}_\gp}

\nc{\parr}{\rm Par}
\nc{\wpar}{\rm WPar}
\nc{\wcomp}{\large{\VDash}}
\nc{\Ker}{\ker}

\nc{\dth}{d} \nc{\mmbox}[1]{\mbox{\ #1\ }} \nc{\fp}{\mrm{FP}} \nc{\rchar}{\mrm{char}} \nc{\Fil}{\mrm{Fil}} \nc{\Mor}{Mor\xspace} \nc{\gmzvs}{gMZV\xspace} \nc{\gmzv}{gMZV\xspace} \nc{\mzv}{MZV\xspace} \nc{\mzvs}{MZVs\xspace} \nc{\Hom}{\mrm{Hom}} \nc{\id}{\mrm{id}} \nc{\im}{\mrm{im}} \nc{\incl}{\mrm{incl}}  \nc{\mchar}{\rm char}

\nc{\Alg}{\mathbf{Alg}} \nc{\Bax}{\mathbf{Bax}} \nc{\bff}{\mathbf f} \nc{\bfk}{{\bf k}} \nc{\bfone}{{\bf 1}} \nc{\bfx}{\mathbf x} \nc{\bfy}{\mathbf y}
\nc{\base}[1]{\bfone^{\otimes ({#1}+1)}} 
\nc{\Cat}{\mathbf{Cat}} \delete{}
\nc{\detail}{\marginpar{\bf More detail}
    \noindent{\bf Need more detail!}
    \svp}
\nc{\Int}{\mathbf{Int}} \nc{\Mon}{\mathbf{Mon}}
\nc{\rbtm}{{shuffle }} \nc{\rbto}{{Rota-Baxter }} \nc{\remarks}{\noindent{\bf Remarks: }} \nc{\Rings}{\mathbf{Rings}} \nc{\Sets}{\mathbf{Sets}}
\nc{\balpha}{\mathbf{\alpha}}

\nc{\BA}{{\mathbb A}} \nc{\CC}{{\mathbb C}} \nc{\DD}{{\mathbb D}} \nc{\EE}{{\mathbb E}} \nc{\FF}{{\mathbb F}} \nc{\GG}{{\mathbb G}} \nc{\HH}{{\mathbb H}} \nc{\LL}{{\mathbb L}} \nc{\NN}{{\mathbb N}} \nc{\KK}{{\mathbb K}} \nc{\PP}{{\mathbb P}} \nc{\QQ}{{\mathbb Q}} \nc{\RR}{{\mathbb R}} \nc{\TT}{{\mathbb T}} \nc{\VV}{{\mathbb V}} \nc{\ZZ}{{\mathbb Z}}

\nc{\cala}{{\mathcal A}} \nc{\calc}{{\mathcal C}} \nc{\cald}{{\mathcal D}} \nc{\cale}{{\mathcal E}} \nc{\calf}{{\mathcal F}} \nc{\calg}{{\mathcal G}} \nc{\calh}{{\mathcal H}} \nc{\cali}{{\mathcal I}} \nc{\call}{{\mathcal L}} \nc{\calm}{{\mathcal M}} \nc{\caln}{{\mathcal N}} \nc{\calo}{{\mathcal O}} \nc{\calp}{{\mathcal P}} \nc{\calr}{{\mathcal R}} \nc{\cals}{{\mathcal S}} \nc{\calt}{{\mathcal T}} \nc{\calw}{{\mathcal W}} \nc{\calk}{{\mathcal K}} \nc{\calx}{{\mathcal X}}
\nc{\calz}{{\mathcal Z}}
\nc{\WC}{\mathcal{WC}}

\nc{\fraka}{{\mathfrak a}} \nc{\frakA}{{\mathfrak A}} \nc{\frakb}{{\mathfrak b}} \nc{\frakB}{{\mathfrak B}}
\nc{\frakc}{{\mathfrak c}}  \nc{\frakD}{{\mathfrak D}}
\nc{\frakH}{{\mathfrak H}}
\nc{\frakh}{{\mathfrak h}} \nc{\frakM}{{\mathfrak M}}
\nc{\frakO}{{\mathfrak O}}
\nc{\frakE}{{\mathfrak E}}
\nc{\bfrakM}{\overline{\frakM}} \nc{\frakm}{{\mathfrak m}} \nc{\frakP}{{\mathfrak P}} \nc{\frakN}{{\mathfrak N}} \nc{\frakp}{{\mathfrak p}} \nc{\frakS}{{\mathfrak S}}
\nc{\frakk}{{\mathfrak k}}
\nc{\frakx}{{\mathfrak x}}
\nc{\frakl}{{\mathfrak l}} \nc{\ox}{\bar{\frakx}} \nc{\frakX}{{\mathfrak X}} \nc{\fraky}{{\mathfrak y}} \nc\dop{\delta}
\nc{\Reduce}{{\rm Red}}

\font\cyr=wncyr10 \font\cyrs=wncyr7
\nc{\redt}[1]{\textcolor{red}{#1}}
\nc{\li}[1]{\textcolor{red}{\tt Li:#1}}
\nc{\yu}[1]{\textcolor{blue}{\tt Yu:#1}}

\begin{document}
\title[Weak composition quasi-symmetric functions, Rota-Baxter algebras and Hopf algebras]{Weak composition quasi-symmetric functions, Rota-Baxter algebras and Hopf algebras}

\author{Li Guo}
\address{
Department of Mathematics and Computer Science, Rutgers University, Newark, NJ 07102, USA}
\email{liguo@rutgers.edu}

\author{Jean-Yves Thibon}
\address{Laboratoire d'Informatique Gaspard Monge, Universit\'e Paris-Est Marne-la-Vall\'ee, 5 Boulevard Descartes,
Champs-sur-Marne, 77454 Marne-la-Vall\'ee cedex 2, France}
\email{jyt@univ-mlv.fr}

\author{Houyi Yu}
\address{School  of Mathematics and Statistics, Southwest University, Chongqing, China}
\email{yuhouyi@swu.edu.cn}

\hyphenpenalty=8000
\date{\today}

\begin{abstract}

We introduce the Hopf algebra of quasi-symmetric functions with semigroup exponents generalizing the Hopf algebra $\qsym$  of quasi-symmetric functions.
As a special case we obtain the Hopf algebra $\Wcqsym$ of weak composition quasi-symmetric functions, which
provides a framework for the study of a question proposed by G.-C.~Rota
relating symmetric type functions and Rota-Baxter algebras.
We provide the transformation formulas between the weak composition monomial and
fundamental quasi-symmetric functions, which extends the corresponding results for quasi-symmetric functions. Moreover, we show that $\qsym$ is a Hopf subalgebra and a Hopf quotient
algebra of $\Wcqsym$. Rota's question is addressed by identifying $\Wcqsym$ with the free commutative unitary Rota-Baxter algebra  $\sha(x)$ of weight 1 on generator $x$, which also allows us to equip $\sha(x)$ with a Hopf algebra structure.

\end{abstract}

\subjclass[2010]{05E05,16W99,16T33}

\keywords{Symmetric functions, quasi-symmetric functions, weak compositions, Rota-Baxter algebras, Hopf algebras}

\maketitle

\tableofcontents

\hyphenpenalty=8000 \setcounter{section}{0}


\allowdisplaybreaks

\section{Introduction}\label{sec:int}

We continue the study from \cite{Ygz2016} to address a question of Rota ~\cite{Ro2} on of the relationship between symmetric related functions, especially quasi-symmetric functions, and Rota-Baxter algebras.
In the present paper we focus on the free commutative unitary Rota-Baxter algebras of weight $1$ generated by one element and weak
composition quasi-symmetric functions, a generalization of quasi-symmetric functions.

As a generalization of the algebra of symmetric functions, the algebra $\qsym$ of quasi-symmetric functions was introduced by Gessel~\cite{Ge} in 1984 to deal with the combinatorics of P-partitions and  the counting
of permutations with given descent sets~\cite{Sta2,St1972}.
Most of the studies on quasi-symmetric functions were carried out after the middle 1990s.
Since then quasi-symmetric functions have grown in importance, interacting with many areas in mathematics including
Hopf algebras \cite{Ehr,MR}, discrete geometry \cite{BHW}, representation theory \cite{Hi2} and algebraic topology \cite{BR2008}.
Generalizations and extensions of quasi-symmetric functions have also been introduced, see for example \cite{AS2005,NCS,HK,NTT}. Further details on quasi-symmetric functions
can be found in the monograph \cite{LMW} and the references therein.

The study of Rota-Baxter algebras (called Baxter algebras in the early literature) originated from the study of Baxter~\cite{Ba} in 1960 from his probability study to understand Spitzer's
identity in fluctuation theory. Formulated formally by Rota and his school in the 1960s \cite{Ro1}, a Rota-Baxter algebra
is an associative algebra equipped with a linear operator that generalizes the integral operator in analysis.
Recently, several interesting developments of Rota-Baxter algebras have been made,
with applications in diverse areas in mathematics and theoretical physics,
such as Hopf algebras \cite{EG2006}, operads \cite{Ag2000}, combinatorics \cite{EGP2007}, quantum field theory \cite{CK2000},
number theory \cite{GZ2008} and Yang-Baxter equations \cite{Bai2007}. See \cite{Gub}, as well as \cite{KRY2009}, for a more detailed introduction to this subject.

The first link between symmetric functions and Rota-Baxter algebras was established a long time ago when Rota \cite{Ro1} gave the first explicit construction of free commutative nonunitary Rota-Baxter algebras. He applied this structure to show that Spitzer's identity that he established for Rota-Baxter algebras
has one of its incarnation as the following Waring's identity relating power sum and elementary symmetric functions:
\begin{equation*}
\exp\left (-\sum_{k=1}^\infty (-1)^kt^k p_k(x_1,x_2,\cdots,x_m)/k
    \right) = \sum_{n=0}^\infty e_n(x_1,x_2,\cdots,x_m)t^n
\text{ for all }\ m\geq 1,
\end{equation*}
where
$$
p_k(x_1,x_2,\cdots,x_m)=x_1^k+x_2^k+\cdots+x_m^k,\quad k\geq1
$$
and
$$
e_n(x_1,x_2,\cdots,x_m)=\sum_{1\leq i_1<i_2<\cdots<i_n\leq m}x_{i_1}x_{i_2}\cdots x_{i_n},\quad  n\geq1
$$
are the {\bf power sum symmetric functions} and {\bf elementary symmetric functions}, respectively, in the polynomial ring  $\QQ[x_1,\cdots,x_m]$,
with the convention that $e_0(x_1,x_2,\cdots,x_m)=1$ and $e_n(x_1,x_2,\cdots,x_m)=0$ if $m< n$.

Motivated by such links between symmetric functions and Rota-Baxter algebras, Rota conjectured~\cite{Ro2}
\begin{quote}
a very close relationship exists between the Baxter identity and the algebra of symmetric functions.
\end{quote}
and concluded
\begin{quote}
The theory of symmetric functions of vector arguments (or Gessel functions) fits nicely with Baxter operators; in fact, identities for such functions easily translate into identities for Baxter operators. $\cdots$
In short: Baxter algebras represent the ultimate and most natural generalization of the algebra of symmetric functions.
\end{quote}

The connection of Rota-Baxter algebras and generalized symmetric functions \cite{Ge1987}
envisioned by Rota turned out to be related to another construction by
Gessel \cite{Ge}, the quasi-symmetric functions.
This relationship was gradually established in the following years. First~\cite{EG2006} proved the equivalence of the mixable shuffle product~\cite{G-K1} (also known as the stuffle product and overlapping shuffle product~\cite{Ha}, among others) in
a free commutative Rota-Baxter algebra and the quasi-shuffle product~\cite{Ho} generalizing quasi-symmetric functions. This realized the algebra of quasi-symmetric functions as a large part of a free commutative Rota-Baxter algebra of weight $1$ on one generator and thus equipped this part of the Rota-Baxter algebra with a Hopf algebra structure. See Section~\mref{sec:CBRBAQSYM} for details.

To relate the full free commutative {\em nonunitary} \rbto algebra of weight $1$ with quasi-symmetric functions, the authors of \cite{Ygz2016} introduced the concept of left weak composition (LWC) quasi-symmetric functions, power series which generalize quasi-symmetric functions with analogous properties. They then realized the free commutative {\em nonunitary} \rbto algebra on one generator as the subalgebra LWCQSym of
LWC quasi-symmetric functions.

There the critical step is to
realize an element of the free commutative nonunitary Rota-Baxter algebra $\sha(x)^0$ as a formal power series.
This is achieved by the correspondence which takes a basis element of $\sha(x)^0$, which is in the form of a pure tensor $x^{\alpha_0}\otimes x^{\alpha_1}\otimes\cdots\otimes x^{\alpha_k}$ indexed by a left weak composition $(\alpha_0,\alpha_1,\cdots,\alpha_k)$ (namely  $\alpha_0,\cdots,\alpha_{k-1}\geq 0, \alpha_k\geq1$), and sends it to the power series $x_0^{\alpha_0} M_{(\alpha_1,\cdots,\alpha_k)}$. Here $M_{(\alpha_1,\cdots,\alpha_k)}$ is the generalized monomial quasi-symmetric function
$$
M_{(\alpha_1,\cdots,\alpha_k)}:= \sum_{1\leq i_1<\cdots<i_k} x_{i_1}^{\alpha_1}\cdots x_{i_k}^{\alpha_k}.
$$
Using this approach, it was shown that the linear span of these power series in $\bfk[[x_0,x_1,x_2,\cdots]]$ is a subalgebra isomorphic to the free commutative nonunitary Rota-Baxter algebra of weight $1$ on one generator.

However, this approach does not work for the full free commutative unitary Rota-Baxter algebra $\sha(x)$ since its basis consists of pure tensors $x^{\alpha_0}\ot \cdots \ot x^{\alpha_k}$ are indexed by all weak compositions, not just left weak compositions and the above correspondence is no longer well-defined.
As a simple example, take the element $x\otimes 1=x\ot x^0$ of $\sha(x)$, indexed by the weak composition $\alpha:=(1,0)$. Then it should correspond to
$x_0 M_{(0)}=x_0 \sum_{n\geq 1} x_n^0=x_0\sum_{n\geq 1}1$
which does not make sense. The same problem arises as long as $\alpha$ ends with a zero.
Thus in order to further investigate the relationship between the full free commutative unitary Rota-Baxter algebra and the algebra of quasi-symmetric functions,
we need to look for a context that is more general than formal power series but still share similar properties in order to define quasi-symmetric functions.

Notice that the definition of mixable shuffle products makes sense for any semigroup and not just for the semigroup of natural numbers,
so a possible context to generalize the quasi-symmetric functions is formal power series
with  suitable semigroup exponents.
As it turns out, quasi-symmetric functions with semigroup exponents, when the semigroup can be
embedded into $\NN^r$, have been considered in~\cite{NPT} to explain the isomorphism between shuffle and quasi-shuffle algebras
and deal with Ecalle's formalism of moulds.
We generalizes the results in \cite{NPT} by taking the semigroup to be any additively finite semigroup without zero-divisors.
In particular, when the semigroup is taken to be the monoid $\tilde\NN$ obtaining from the additive monoid $\NN$ of nonnegative integers by adding an extra element $\varepsilon$,
we obtain the algebra of weak composition quasi-symmetric functions ($\Wcqsym$ for short),
which will enable us to interpret the full free commutative unitary Rota-Baxter algebra as a suitable generalization of quasi-symmetric functions.
The interesting point here is the fact that the subsemigroup $\tilde\NN\backslash\{0\}$ is isomorphic to the additive monoid $\NN$,
which induces a bijection between the set of $\tilde\NN$-compositions and the set of weak compositions, so that we have
the term of ``weak composition quasi-symmetric functions".

The outline of this paper is as follows. In Section \ref{sec:unifrba},
we introduce the definitions of formal power series algebras and quasi-symmetric functions with semigroup exponents, and then explore the Hopf algebra structure of the algebra of
quasi-symmetric functions with semigroup exponents.
By taking the semigroup to be $\tilde{\NN}$, in Section \ref{sec:Eqsyms}, we focus on the algebra $\Wcqsym$,
extending the notation of quasi-symmetric functions as a special case of quasi-symmetric functions with semigroup exponent,
and investigate some properties of $\Wcqsym$. More precisely, we first develop the monomial and fundamental bases for $\Wcqsym$ respectively  and establish the transformation formulas for them,
generalizing the corresponding results for quasi-symmetric functions. Then we show that  $\qsym$ is both a Hopf subalgebra
and a Hopf quotient algebra of $\Wcqsym$.
The $\Wcqsym$ permits us to put the free commutative unitary Rota-Baxter algebra $\sha(x)$ on $x$ in the setting of quasi-symmetric functions in Section \ref{sec:CBRBAQSYM}, addressing the question of Rota quoted at the begining of the introduction.
In particular, we use these connections to obtain a Hopf algebra structure on $\sha(x)$, completing the previous efforts~\cite{AGKO,EG2006} on this subject.
\smallskip

\noindent
{\bf Convention.} Unless otherwise specified, an algebra in this paper is assumed to be commutative, defined over a commutative ring $\bfk$ containing $\QQ$ with characteristic $0$.  By a tensor product we mean the tensor product over $\bfk$. Let $\NN$ and $\PP$ denote the set of nonnegative and positive integers respectively.

\section{\Sgqsyms}
\mlabel{sec:unifrba}

In this section we generalize the notion of quasi-symmetric functions to quasi-symmetric functions with semigroup exponents.
When the semigroup is taken to be the additive semigroup $\NN$ of nonnegative integers, we recover the classical quasi-symmetric functions.

\subsection{Formal power series algebras with semigroup exponents}
To begin with, let us generalize the formal power series algebra.

A formal power series is a (possibly infinite) linear combination of monomials $x_{i_1}^{\alpha_1}x_{i_2}^{\alpha_2}\cdots x_{i_k}^{\alpha_k}$ where $\alpha_1,\alpha_2,\cdots,\alpha_k$ are positive integers, which can be regarded as the locus of the map from $X:=\{x_n\,|\,n\geq 1\}$ to $\NN$ sending $x_{i_j}$ to $\alpha_j$, $1\leq j\leq k$, and everything else in $X$ to zero.
Our generalization of the formal power series algebra is simply to replace $\NN$ by a suitable additive monoid with a zero element.

\begin{defn}
Let $B$ be a commutative additive monoid with zero $0$ such that $B\backslash\{0\}$ is a subsemigroup. Let $X$ be a finite or countably infinite totally ordered set of
commutating variables.
The set of {\bf $B$-valued maps} is defined to be
\begin{equation}\label{eq:bsemi}
\map{X}{B}:=\left\{ f:X\to B\,|\, \supp (f) \text{ is finite }\right\},
\end{equation}
where $\supp(f):=\{x\in X\,|\, f(x)\neq 0\}$ denotes the support of $f$.
\end{defn}
The addition on $B$ equips $\map{X}{B}$ with an addition by
$$ (f+g)(x):=f(x)+g(x)\quad  \text{ for all } f, g\in B^X \text{ and } x\in X,$$
making $\map{X}{B}$ into an additive monoid.
Resembling the formal power series, we identify $f\in \map{X}{B}$ with its locus $\{(x,f(x))\,|\,x\in \supp(f)\}$ expressed in the form of a formal product
\begin{align*}
X^f:=\prod_{x\in X}x^{f(x)}=\prod_{x\in \supp(f)}x^{f(x)},
\end{align*}
called a {\bf $B$-exponent monomial}, with the convention $x^0=1$.

By an abuse of notation, the addition on $\map{X}{B}$ becomes
\begin{equation}
X^fX^g=X^{f+g}\quad \text{ for all } f, g\in \map{X}{B}.
\mlabel{eq:xmap}
\end{equation}
We then form the semigroup algebra
$$ \bfk [X]_B:=\bfk \map{X}{B}$$
consisting of linear combinations of $\map{X}{B}$, called the algebra of {\bf $B$-exponent polynomials}.
Similarly, we can define the free $\bfk$-module $\bfk[[X]]_B$ consisting of possibly infinite linear combinations of $\map{X}{B}$, called {\bf $B$-exponent formal power series}.
If $B$ is {\bf additively finite} in the sense that for any $a\in B$ there are finite number of pairs $(b,c)\in B^2$ such that $b+c=a$,
then the multiplication in Eq.~(\mref{eq:xmap}) extends by bilinearity to a multiplication on $\bfk[[X]]_B$, making it into a $\bfk$-algebra, called the
{\bf algebra of formal power series with semigroup  $B$-exponents}.

Let $B$ be a finitely generated free commutative  additively finite monoid with generating set $\{b_1,b_2,\cdots,b_t\}$.
Then
\begin{align*}
\bfk [X]_B=\bfk[x^{b_i}|1\leq i\leq t,x\in X].
\end{align*}
For example, taking $B$ as the additive monoid $\NN$ of nonnegative integers, then $\map{X}{B}$ is simply the free monoid generated by $X$ and $\bfk[X]_B$ is
the free commutative algebra $\bfk[X]$.

\subsection{Quasi-symmetric functions with semigroup exponents}\label{sec:qsymwse}

We now generalizes the quasi-symmetric functions to the context of formal power series with semigroup exponents.
See~\cite{NPT} for semigroup exponent quasi-symmetric functions in the study of moulds.

Let $B$ be a commutative additively finite monoid with zero $0$ such that $B\backslash\{0\}$ is a subsemigroup, and let $b\in B$.
A {\bf weak $B$-composition} of $b$ is a finite sequence $\mathbf{\alpha}=(\alpha_1,\alpha_2,\cdots,\alpha_k)$ of elements of $B$ which sum to $b$.
We call the $\alpha_i$ for $1\leq i\leq k$ the {\bf entries} of $\alpha$ and $\ell(\alpha):=k$ the {\bf length} of $\alpha$.
The {\bf weight} of a weak $B$-composition $\alpha$, denoted by $|\alpha|$, is the sum of its entries.
By convention we denote by $\emptyset$ the unique weak $B$-composition whose weight and length are $0$, called the {\bf empty weak $B$-composition}.
We let $\WC(B)$ denote the set of all weak $B$-compositions.

A {\bf $B$-composition} $\alpha$ of a non-zero element $b\in B$ is a finite-ordered list of non-zero elements whose sum is $b$.
Thus every $B$-composition is a weak $B$-composition, but not vice-versa. For convenience, the empty weak $B$-composition $\emptyset$ is also called the {\bf empty $B$-composition}.
We denote the set of $B$-compositions of $b$ by $\mathcal{C}(B,b)$, and write $\mathcal{C}(B):=\bigcup\limits_{0\neq b\in B}\mathcal{C}(B,b)$.

Given a weak $B$-composition $\alpha$, the {\bf reversal} of $\alpha$, denoted by $\alpha^{r}$, is obtained by writing the entries of $\alpha$ in the reverse order.
For a pair of weak $B$-compositions $\alpha=(\alpha_1,\cdots,\alpha_k)$ and $\beta=(\beta_1,\cdots,\beta_{\ell})$, the {\bf concatenation}
of $\alpha$ and $\beta$ is
\begin{align*}
\alpha\cdot\beta:=(\alpha_1,\cdots,\alpha_k,\beta_1,\cdots,\beta_{\ell}).
\end{align*}

When the monoid $B$ is taken to be the commutative additive monoid $\NN$ of nonnegative integers, we obtain the definition of
{\bf weak compositions} and {\bf compositions} \cite{Sta} in the literature,
which will be called weak $\NN$-compositions and $\NN$-compositions respectively in what follows for the sake of clarity. Furthermore, we
write $\alpha\models n$ if $\alpha$ is an $\NN$-composition of $n$.

The refining order defined as follows plays an important role in the theory of quasi-symmetric functions.
Let $n$ be a positive integer. Given an $\NN$-composition $\alpha=(\alpha_1,\alpha_2,\cdots,\alpha_k)\models n$, define its associated {\bf descent set}
\begin{align*}
{\rm set}(\alpha)=\{\alpha_1,\alpha_1+\alpha_2,\cdots,\alpha_1+\alpha_2+\cdots+\alpha_{k-1}\}\subseteq[n-1],
\end{align*}
where $[n]$ is the set $\{1,2,\cdots,n\}$ for any nonnegative integer $n$.
This gives a bijection between the set $\mathcal{C}(\NN,n)$ of all $\NN$-compositions of $n$
and the subsets of $[n-1]$.
The {\bf refining order}, denoted by $\preceq$, on $\mathcal{C}(\NN,n)$ is defined by
\begin{align*}
\alpha\preceq\beta \quad {\rm if\ and\ only\ if}\ {\rm set}(\beta)\subseteq{\rm set}(\alpha).
\end{align*}

For example, if $\alpha=(1,3,2)$, $\beta=(4,2)$, $\gamma=(6,5)$, then $\alpha\models6$, $\beta\models6$, $\text{set}(\alpha)=\{1,4\}$ and $\text{set}(\beta)=\{4\}$ so that
$\alpha\preceq \beta$. Moreover, $\alpha^r=(2,3,1)$ and $\alpha\cdot\gamma=(1,3,2,6,5)$.

\begin{defn}\label{def:bquasisymm} Let $B$ be a commutative additively finite monoid with zero $0$ such that $B\backslash\{0\}$ is a subsemigroup,
and let $X=\{x_1<x_2<\cdots\}$ be an ordered set of
mutually commuting variables. Consider the formal power series algebra $\bfk [[X]]_B$
over $\bfk$.
A formal power series $f\in \bfk [[X]]_B$ is called a {\bf $B$-quasi-symmetric function} if, for any $B$-composition $(\alpha_1,\alpha_2,\cdots,\alpha_k)$, the coefficients of
$y_1^{\alpha_1}y_2^{\alpha_2}\cdots y_k^{\alpha_k}$
and $z_1^{\alpha_1}z_2^{\alpha_2}\cdots z_k^{\alpha_k}$ in $f$ are equal for all totally ordered subsets of indeterminates  $y_1<y_2<\cdots<y_k$
and $z_1<z_2<\cdots<z_k$. We denote the set of all $B$-quasi-symmetric functions  by $\qsym(X)_B$, or $\qsym_B$ for short.
\end{defn}

Analogous to the monomial basis for quasi-symmetric functions,
for a $B$-composition $\mathbf{\alpha}=(\alpha_1,\alpha_2,\cdots,\alpha_k)$, consider the {\bf monomial $B$-quasi-symmetric function}
\begin{align}\label{eq:quasimf}
M_{\mathbf{\alpha}}=\sum_{1\leq i_1<i_2<\cdots <i_k} x_{i_1}^{\alpha_1}x_{i_2}^{\alpha_2}\cdots x_{i_k}^{\alpha_k}
\end{align}
indexed by the $B$-composition $\mathbf{\alpha}$, with the notation $M_{\emptyset}=1$. If $\alpha$ is a nonzero element of $B$, then we write $M_{(\alpha)}=M_{\alpha}$ for simplicity.
We show that the family of monomial $B$-quasi-symmetric functions forms a basis of $\qsym_B$.

\begin{lemma}\label{lem:bmonomialbasis}
The set $\{M_{\alpha}|\alpha\in \mathcal{C}(B)\}$ is a $\bfk$-basis for $\qsym_B$.
\end{lemma}
\begin{proof}
By Definition \ref{def:bquasisymm}, any $B$-quasi-symmetric functions $f\in\qsym_B$ can be written as a $\bfk$-linear combination of $\{M_{\alpha}|\alpha\in \mathcal{C}(B)\}$.
Hence we need to show that $\{M_{\alpha}|\alpha\in \mathcal{C}(B)\}$ is linear independent.
Assume that $\sum_{\alpha\in \Lambda}c_{\alpha}M_{\alpha}=0$ where $\Lambda$ is a finite set of $B$-compositions and $c_{\alpha}\in\bfk$ for all $\alpha\in \Lambda$.
Notice that $\qsym_B\subseteq \bfk [[X]]_B$ and that the set of all $B$-exponent monomials forms a $\bfk$-basis for $\bfk [[X]]_B$.
So considering the expression Eq.~\eqref{eq:quasimf} for each $M_{\alpha}$,  we must have
\begin{align*}
\sum_{\alpha\in \Lambda}c_\alpha x_1^{\alpha_1}x_2^{\alpha_2}\cdots x_{\ell{(\alpha)}}^{\alpha_\ell{(\alpha)}}=0,
\end{align*}
and hence $c_{\alpha}=0$ for all $\alpha\in \Lambda$. This completes the proof.
\end{proof}

Therefore, if we define the degree of $M_\alpha$ to be the weight of $\alpha$ for a  $B$-composition $\alpha$,  then  $\qsym_B$ is a $B$-graded free $\bfk$-module,
denoted by $\qsym_B=\bigoplus_{b\in B}\qsym_B^b$, where $\qsym_B^b$ is the free $\bfk$-module spanned by $\{M_\alpha|\alpha\in \mathcal{C}(B,b)\}$.
We next show that $\qsym_B$ is closed
under the natural product of formal power series and moreover  that
the multiplication rule of two monomial $B$-quasi-symmetric functions is dictated by the quasi-shuffle product defined by the following recursion.

Let  $\bfk\mathcal{C}(B)=\bigoplus_{\alpha\in \mathcal{C}(B)}\bfk\alpha$ be the free $\bfk$-module generated by
the set of all $B$-compositions. For simplicity,
if $a\in B$ and  $\alpha=(\alpha_1,\alpha_2,\cdots,\alpha_k)\in\mathcal{C}(B)$, then we write $(a,\alpha)=(a,\alpha_1,\alpha_2,\cdots,\alpha_k)$ for short.
Now define the {\bf quasi-shuffle product}, denoted by $*$, on $\bfk\mathcal{C}(B)$ by requiring that $\emptyset*\alpha=\alpha*\emptyset=\alpha$
for any $B$-composition $\alpha$, and that,
for any $B$-compositions $\alpha$, $\beta$ and $a,b\in B$,
\begin{align}\label{eq:compoverlapshuprod}
(a,\alpha)*(b,\beta)=(a,\alpha*(b,\beta))+(b,(a,\alpha)*\beta)+(a+b,\alpha*\beta)).
\end{align}

\begin{prop}\label{prop:kcstoqsyms}
For any $\alpha,\beta\in \mathcal{C}(B)$, we have $M_{\alpha}M_{\beta}=\sum\limits_{\gamma\in \mathcal{C}(B)}\langle\gamma, \alpha*\beta\rangle M_{\gamma}$,
where $\langle\gamma, \alpha*\beta\rangle$ is the coefficient of the $B$-composition $\gamma$ in $\alpha*\beta$.
In other words, the assignment $\alpha\mapsto M_{\alpha}$ defines a homogeneous isomorphism from the quasi-shuffle algebra $\bfk\mathcal{C}(B)$ to $\qsym_B$.
\end{prop}

Thus, we will write $M_{\alpha}M_{\beta}=M_{\alpha*\beta}$, for simplicity.
\begin{proof}
Let $\alpha=(\alpha_1,\alpha_2,\cdots,\alpha_k)$, $\beta=(\beta_1,\beta_2,\cdots,\beta_l)$ be $B$-compositions. Then, by Eq.~\eqref{eq:quasimf}, we have
\begin{align}\label{eq:malmbe=quasisproduct}
M_{\alpha}M_{\beta}=\sum_{\substack{{1\leq n_1<n_2<\cdots<n_k}\\{1\leq m_1<m_2<\cdots<m_l}}}
x_{n_1}^{\alpha_1} x_{n_2}^{\alpha_2}\cdots x_{n_k}^{\alpha_k} x_{m_1}^{\beta_1} x_{m_2}^{\beta_2}\cdots x_{m_l}^{\beta_l}.
\end{align}
The proof now follows by induction on $k+l$. If $k+l=0$, then $k=l=0$ so that $\alpha=\beta=\emptyset$ and hence $M_{\alpha}=M_{\beta}=1$,
so the assertion is true. Now assume that the desired identity holds for $k+l\leq s$ for a given $s\geq 0$ and consider the case $k+l=s+1$.
Then, by comparing the sizes of $n_1$ and $m_1$, we have
\begin{align*}
M_{\alpha}&M_{\beta}=\sum_{1\leq n_1}x_{n_1}^{\alpha_1}\sum_{\substack{{n_1<n_2<\cdots<n_k}\\{n_1< m_1<m_2<\cdots<m_l}}}
 x_{n_2}^{\alpha_2}\cdots x_{n_k}^{\alpha_k} x_{m_1}^{\beta_1} x_{m_2}^{\beta_2}\cdots x_{m_l}^{\beta_l}+\\
 &\sum_{1\leq m_1}x_{m_1}^{\beta_1}\sum_{\substack{{m_1< n_1<n_2<\cdots<n_k}\\{m_1<m_2<\cdots<m_l}}}
 x_{n_2}^{\alpha_2}\cdots x_{n_k}^{\alpha_k} x_{m_1}^{\beta_1} x_{m_2}^{\beta_2}\cdots x_{m_l}^{\beta_l}
 +\sum_{1\leq n_1=m_1}x_{n_1}^{\alpha_1+\beta_1}\sum_{\substack{{n_1<n_2<\cdots<n_k}\\{n_1<m_2<\cdots<m_l}}}
 x_{n_2}^{\alpha_2}\cdots x_{n_k}^{\alpha_k} x_{m_1}^{\beta_1} x_{m_2}^{\beta_2}\cdots x_{m_l}^{\beta_l}.
\end{align*}
By the induction hypothesis, we obtain that
$$M_{\alpha}M_{\beta}=M_{(\alpha_1,(\alpha_2,\alpha_3,\cdots,\alpha_k)*\beta)}+M_{(\beta_1,\alpha*(\beta_2,\beta_3,\cdots,\beta_l))}+
M_{(\alpha_1+\beta_1,(\alpha_2,\alpha_3,\cdots,\alpha_k)*(\beta_2,\beta_3,\cdots,\beta_l))},$$
that is, $M_{\alpha}M_{\beta}=\sum\limits_{\gamma\in \mathcal{C}(B)}\langle\gamma, \alpha*\beta\rangle M_{\gamma}$, as desired.
\end{proof}

Thus sums and products of $B$-quasi-symmetric functions are again $B$-quasi-symmetric. In other words, the set  $\qsym_B$ of all $B$-quasi-symmetric functions forms
a subalgebra of $\bfk [[X]]_B$, which is called the {\bf algebra of quasi-symmetric functions with semigroup $B$-exponents}.

We remark that it is certainly not the case that $\qsym_B$ is an algebra for all commutative additive monoids with zero $0$.
For example, if $B\backslash\{0\}$ is not a subsemigroup of $B$, then there exist
$a,b\in B\backslash\{0\}$ such that $a+b=0$, so one has $M_aM_b=M_{(a,b)}+M_{(b,a)}+M_{0}$. However, we would have
$ M_{0}=\sum_{i} x_i^0$ which does not make sense by definition.

When the semigroup $B$ is specialized to the commutative additive monoid  $\NN$ of nonnegative integers, we obtain the  algebra $\qsym$ of quasi-symmetric functions.
Let $\qsym_n$ denote the space of homogeneous
quasi-symmetric functions of degree $n$, then $\qsym=\bigoplus_{n\geq0}\qsym_n$.
Here $\qsym_0$ is spanned by $M_{\emptyset}=1$, and for each $n\geq1$, $\qsym_n$ has a natural {\bf monomial basis}, given by the set of all $M_{\alpha}$ for $\alpha=(\alpha_1,\alpha_2,\cdots,\alpha_k)\models n$, where
\begin{align*}
M_{\alpha}:=\sum_{1\leq i_1<i_2<\cdots<i_k}x_{i_1}^{\alpha_1}x_{i_2}^{\alpha_2}\cdots x_{i_k}^{\alpha_k}.
\end{align*}

\subsection{Hopf algebra structure of $\qsym_B$}
In this subsection, we will show that $\qsym_B$ forms a $B$-graded Hopf algebra for any
commutative additively finite monoid $B$ with zero $0$ such that $B\backslash\{0\}$ is a subsemigroup.
As in the case of $\qsym$, a coproduct $\Delta_B$ and a counit $\epsilon_B$ can be defined on $\qsym_B$ by the following formulas on the monomial basis elements:
\begin{align}\label{mqsydeltcop1}
\Delta_B(M_{\alpha})=\sum_{\alpha=\beta\cdot \gamma}M_\beta\otimes M_\gamma=\sum_{i=0}^kM_{(\alpha_1,\cdots,\alpha_i)}\otimes M_{(\alpha_{i+1},\cdots,\alpha_k)},
\end{align}
\begin{align}\label{mqsydeltcop2}
\epsilon_B(M_{\alpha})=&\delta_{\alpha,\emptyset},
\end{align}
where $\alpha=(\alpha_1,\cdots,\alpha_k)$ is a $B$-composition.
It is easy to see that the coproduct is coassociative. The fact that both the coproduct and the counit
are algebra homomorphisms can be proved analogously to~\cite[Theorem 3.1]{Ho}, making $\qsym_B$ into a bialgebra.

Furthermore, $\qsym_B$  is a $B$-graded bialgebra. More precisely, we have  $\qsym_B=\bigoplus_{b\in B}\qsym_B^b$  and
\begin{align*}
(\qsym_B^b)(\qsym_B^c)\subseteq \qsym_B^{b+c},\qquad
\Delta_B(\qsym_B^a)\subseteq \bigoplus_{b+c=a}\qsym_B^b\otimes \qsym_B^c
\end{align*}
for all $a,b,c\in B$.

However the fact that any connected graded bialgebra is naturally a Hopf algebra \cite{Tak1971} does not apply here since the grading is not $\NN$-graded.
Therefore, to show that the bialgebra $\qsym_B$ admits the structure of a Hopf algebra,
we will show directly that the antipode exists.
For this purpose, we generalize the refining order $\preceq$ on the set $\mathcal{C}(\NN,n)$ of $\NN$-compositions of $n$ to an order, denoted by $\leq$,
on the set $\mathcal{C}(B,b)$ of all $B$-compositions of $b$ where $0\neq  b\in B$.
Let $\alpha=(\alpha_1,\cdots,\alpha_k)$ be a $B$-composition, and let $J=(j_1,\cdots,j_l)$ be an $\NN$-composition of $k$, the length of $\alpha$.
The $B$-composition $J\circ\alpha$ is defined by
\begin{align*}
J\circ\alpha=(\alpha_1+\cdots+\alpha_{j_1},\alpha_{j_1+1}+\cdots+\alpha_{j_1+j_2},\cdots,\alpha_{j_1+j_2+\cdots+j_{l-1}+1}+\cdots+\alpha_{k}).
\end{align*}
For two $B$-compositions $\alpha,\beta$ of the same weight, if there exists an $\NN$-composition $J\models \ell(\alpha)$ such that $\beta=J\circ \alpha$, then we say $\alpha\leq\beta$.
Generally speaking, the $\NN$-composition $J$, if any, such that $\beta=J\circ \alpha$ is not unique, since $B$ may not be a cancelative monoid (that is, $a+c=b+c \Rightarrow a=b$).
For example, let $B$ be a left zero semigroup with zero,  that is, $a+b=a$ for all nonzero elements $a,b\in B$, and let $\alpha=(a,b,a,a,c)$, $\beta=(a,a)$ be $B$-compositions.
Then we have $\alpha\leq (2,3)\circ\alpha=\beta$. On the other hand, we also have
$\alpha\leq (3,2)\circ\alpha=\beta$.

It is obvious that the partial order $\leq$ on the set $\mathcal{C}(B,b)$ of $B$-compositions of $b$ is generated by  the covering relation
\begin{align*}
(\alpha_1,\cdots,\alpha_{i},\alpha_{i+1},\cdots,\alpha_n)\leq(\alpha_1,\cdots,\alpha_{i}+\alpha_{i+1},\cdots,\alpha_n).
\end{align*}
In other words, if $\alpha\leq \beta$, then we can obtain the entries of $\beta$ by
adding together adjacent entries of $\alpha$.

\begin{prop}\label{antieqsym}
Let $B$ be a commutative additively finite monoid with zero $0$ such that $B\backslash\{0\}$ is a subsemigroup.
Then $\qsym_B$ is a Hopf algebra, where the antipode $S_B$ is given by
\begin{align}\label{antipodeqesym0}
S_B(M_{\alpha})=(-1)^{\ell(\alpha)}\sum_{J\models \ell(\alpha)}M_{J\circ\alpha^r}.
\end{align}
\end{prop}
\begin{proof}
The proof of this statement is quite similar to that of \cite[Theorem 3.2]{Ho}, but we include a detailed proof here for completeness.

It suffices to show that $S_B$ is the antipode since $\qsym_B$ is a bialgebra.
With Sweedler's sigma notation \cite{Swe},
it suffices to show that the  linear map $S_B:\qsym_B\rightarrow \qsym_B$ satisfies the condition
\begin{align}\label{criteriaantipodeqsym}
\sum_{\alpha}S_B(M_{\alpha_{(1)}})M_{\alpha_{(2)}}=\epsilon_B(M_{\alpha})\cdot1=\sum_{\alpha}M_{\alpha_{(1)}}S_B(M_{\alpha_{(2)}})
\end{align}
for all $B$-compositions $\alpha$, where $\Delta_B(M_{\alpha})=\sum_{\alpha}M_{\alpha_{(1)}}\otimes M_{\alpha_{(2)}}$.
The proof of the two identities are analogous, so we only show the first one.
By Eq.~\eqref{antipodeqesym0}, $S_B(1)=1=\epsilon_B(1)\cdot1$.
Note that $\epsilon_B(M_{\alpha})=0$ if $\alpha\neq\emptyset$, so it is enough to show
\begin{align}\label{eq:antibqsymbamm0}
\sum_{\alpha}S_B(M_{\alpha_{(1)}})M_{\alpha_{(2)}}=0
\end{align}
for all nontrivial $B$-compositions $\alpha$.

We apply the induction on $\ell(\alpha)$, the length of $\alpha$.
If $\ell(\alpha)=1$, then $S_B(M_\alpha)=-M_{\alpha}$ by a simple computation, and hence Eq.~\eqref{eq:antibqsymbamm0} holds.
For $n\geq 2$, supposing that Eq.~\eqref{eq:antibqsymbamm0} holds for all $\alpha\in \mathcal{C}(B)$ with $\ell(\alpha)< n$.
Let $\alpha=(\alpha_1,\alpha_2,\cdots,\alpha_n)$ be a $B$-composition.
Then
\begin{align*}
\sum_{\alpha}S_B(M_{\alpha_{(1)}})M_{\alpha_{(2)}}=&\sum_{k=0}^{n}S_B\left(M_{(\alpha_1,\cdots,\alpha_k)}\right)M_{(\alpha_{k+1},\cdots,\alpha_n)}\cr
=&\sum_{k=0}^{n}(-1)^k\sum_{J\models k}M_{J\circ(\alpha_k,\cdots,\alpha_1)}M_{(\alpha_{k+1},\cdots,\alpha_n)}.
\end{align*}
So it suffices to show
\begin{align}\label{antiqsye2}
(-1)^{n}\sum_{J\models n}M_{J\circ\alpha^r}
=&\sum_{k=0}^{n-1}(-1)^{k+1}\sum_{J\models k}M_{J\circ(\alpha_k,\cdots,\alpha_1)}M_{(\alpha_{k+1},\cdots,\alpha_n)}.
\end{align}
Now the first entry of $\beta$
for each monomial $B$-quasi-symmetric function $M_{\beta}$ occurs in the above expansion on the right hand is one of the following three cases:
$\alpha_k+\cdots+\alpha_j$, $\alpha_{k+1}$, or $\alpha_{k+1}+\alpha_{k}+\cdots+\alpha_{j}$ for some $1\leq j\leq k$. Here a distinction is made between $\alpha_{i_1}+\alpha_{i_1-1}+\cdots+\alpha_{i_2}$ and
$\alpha_{i_1}+\alpha_{i_1-1}+\cdots+\alpha_{i_3}$
for distinct $i_2$ and $i_3$, although they may have the same  value (see Example \ref{antipexameqsym} below).
We say that the term is of type $k$ in the first case, and of  type $k+1$ in the latter two cases.
Now consider a  monomial $B$-quasi-symmetric function  that appears on the right hand of Eq.~\eqref{antiqsye2}.
If it has type $i$ with $1\leq i\leq n-1$, then it will occur for both $k=i$  and $k=i-1$, and the two occurrences will have opposite signs and hence will cancel
each other.
Thus the only monomial $B$-quasi-symmetric functions that do not cancel are those of type $n$, which  will  appear only for
$k=n-1$ and have the coefficient $(-1)^n$. This gives the left hand side of Eq.~\eqref{antiqsye2}, completing the proof.
\end{proof}

\section{Weak composition quasi-symmetric functions}
\mlabel{sec:Eqsyms}

Now we consider the case when the commutative additively finite monoid $B$ is the monoid obtained from the additive monoid $\NN$ of nonnegative integers by adjoining a new element $\varepsilon$,
that is,  $B=\tilde{\NN}:=\mathbb{N}\cup \{\varepsilon\}$,
satisfying $0+\varepsilon=\varepsilon+0=\varepsilon+\varepsilon=\varepsilon$ and $n+\varepsilon=\varepsilon+n=n$ for all $n\geq1$. For convenience we extend the natural order on $\NN$ to $\tilde{\NN}$
by defining $0<\varepsilon<1$.

\subsection{WC monomial  quasi-symmetric functions}\label{sec:wcmqsf}

Recall that an $\tilde{\NN}$-composition is a finite sequence of non-zero elements of $\tilde{\NN}$.
The connection between the algebra  $\qsym_{\tilde{N}}$ of quasi-symmetric functions with exponents in $\tilde{\NN}$ and the set $\WC(\NN)$ of weak compositions is established by the map
\begin{align}
\theta: \tilde{\NN}\backslash \{0\}\longrightarrow \NN, n\mapsto \left\{\begin{array}{ll} 0, & n=\vep, \\ n, & \text{otherwise},
\end{array} \right.
\label{eq:theta1}
\end{align}
which induces the natural bijection
\begin{align}
\theta:\mathcal{C}(\tilde{\NN})\longrightarrow  \WC(\NN),
\alpha \mapsto \left\{\begin{array}{ll} \emptyset, & \alpha=\emptyset,\\ (\theta(\alpha_1),\cdots, \theta(\alpha_k)), & \alpha=(\alpha_1,\cdots,\alpha_k)\in\mathcal{C}(\tilde{\NN}), k\geq 1,\end{array}
\right .
\mlabel{eq:theta2}
\end{align}
which induces by $\bfk$-linearity the linear bijection
\begin{align}\label{thetabijcntowc}
\theta: \bfk\, \mathcal{C}(\tilde{\NN})\longrightarrow  \bfk\,\WC(\NN).
\end{align}
So the effect of $\theta$ is replacing all entries of $\alpha=(\alpha_1,\cdots,\alpha_k)\in \mathcal{C}(\tilde{\NN})$ which are $\vep$ by $0$.

Through the bijection $\theta$,
the quasi-shuffle product $*$ on $\bfk\mathcal{C}(\tilde{\NN})$ defines a product on $\bfk\WC(\NN)$, still denoted $*$, by the transport of structures:
\begin{align*}
\alpha*\beta=\theta(\theta^{-1}(\alpha)*\theta^{-1}(\beta)).
\end{align*}
Hence, $\theta$ is an algebra isomorphism from $(\bfk\mathcal{C}(\tilde{\NN}),*)$ to $(\bfk\WC(\NN),*)$.
Combining with Proposition \ref{prop:kcstoqsyms}, we see that $\qsym_{\tilde{N}}$ is isomorphic to $(\bfk\WC(\NN),*)$.

Based on the above arguments, for an $\tilde{\NN}$-composition $\alpha\in \mathcal{C}(\tilde{\NN})$, the monomial basis $M_{\alpha}$ will be called a {\bf weak composition (WC) monomial quasi-symmetric function}. Moreover, we write $\Wcqsym$ for $\qsym_{\tilde{N}}$, and call it the algebra of {\bf weak composition (WC) quasi-symmetric functions}.

For a given $n\in\tilde{\NN}$, let $\Wcqsym_n$ denote the vector space spanned  by the set of all WC monomial quasi-symmetric functions of degree $n$, that is,
$\Wcqsym_n=\bigoplus_{\alpha\in \mathcal{C}(\tilde{\NN},n)}\bfk M_{\alpha}$, where the degree of a basis element $M_{\alpha}$
is given by the weight $|\alpha|$ of $\alpha$.
Then $\Wcqsym=\bigoplus_{n\in\tilde{\NN}}\Wcqsym_n$ is the graded algebra of WC quasi-symmetric functions.
Note that this is an $\tilde{\NN}$-graded algebra. Since $\varepsilon+n=n$ for any $n\geq\varepsilon$, the dimension of each homogeneous
piece of degree larger than $0$ is infinite.

\subsection{WC fundamental  quasi-symmetric functions}
Besides the monomial basis $M_\alpha$ the homogeneous component $\qsym_n$ of the algebra $\qsym$ has a second important basis known as Gessel's {\bf fundamental quasi-symmetric functions}~\cite{Ge}, also indexed by $\NN$-compositions $\alpha=(\alpha_1,\alpha_2,\cdots,\alpha_k)\models n$, which can be expressed by
\begin{align*}
F_{\alpha}=\sum_{\beta\preceq\alpha}M_{\beta}.
\end{align*}

As in the classical case, we will define WC fundamental  quasi-symmetric functions, which will be reduced to the fundamental quasi-symmetric
functions \cite{Ge}, and provide the transformation formula for the WC monomial  and WC fundamental  quasi-symmetric functions.

By definition, each $\tilde{\NN}$-composition $\alpha$  can be expressed uniquely in the form
\begin{align*}
\alpha=(\varepsilon^{i_1},s_1,\varepsilon^{i_2},s_2,\cdots,\varepsilon^{i_k},s_k,\varepsilon^{i_{k+1}}),
\end{align*}
where $i_1,i_2,\cdots,i_{k+1}\in \NN$, $s_1,s_2,\cdots,s_k\in \PP$ and $\varepsilon^i$ means a string of $i$ components of $\varepsilon$ following~\cite{Mac}.
We denote by $\bar{\alpha}$ the $\NN$-composition obtaining from $\alpha$ by omitting its $\varepsilon$ components, that is,
$\bar{\alpha}=(s_1,s_2,\cdots,s_k)$. Let $\ell_{\vep}(\alpha)$ denote the number of entries in $\alpha$ which are equal to $\varepsilon$. So $\ell_{\varepsilon}(\alpha)=i_1+i_2+\cdots+i_{k+1}$.

For any $\NN$-composition $\sigma=(s_1,s_2,\cdots,s_k)$, the fundamental quasi-symmetric function
indexed by $\sigma$ can be written as
\begin{align*}
F_{\sigma}:=\sum_{\substack{1\leq n_1\leq n_2\leq\cdots\leq n_{|\sigma|}\\ \ell\in \text{set}(\sigma) \Rightarrow n_\ell<n_{\ell+1} }} x_{n_1}x_{n_2}\cdots x_{n_{|\sigma|}}.
\end{align*}
More generally, for an $\tilde{\NN}$-composition $\alpha=(\varepsilon^{i_1},s_1,\varepsilon^{i_2},s_2,\cdots,\varepsilon^{i_k},s_k,\varepsilon^{i_{k+1}})$,
where $i_1,i_2,\cdots,i_{k+1}\in\NN$ and $s_1,s_2,\cdots,s_k\in \PP$,
we denote $a_j=i_1+s_1+\cdots+i_j+s_j$, $j=1,2,\cdots,k$, and denote $\text{set}(\alpha)=\{a_1,\cdots,a_k\}$. We then define the {\bf WC fundamental  quasi-symmetric function} indexed by $\alpha$ to be the
formal power series
\begin{align}\label{fundmbasis1}
F_{\alpha}:=\sum_{\substack{n_1\leq n_2\leq \cdots \leq n_{a_{k}+i_{k+1}} \\ \ell\in \text{set}(\alpha) \Rightarrow n_{\ell}<n_{\ell+1}}}
x_{n_1}^\varepsilon\cdots x_{n_{i_1}}^\varepsilon x_{n_{i_1+1}}\cdots x_{n_{a_1}}
\cdots x_{n_{a_{k-1}+1}}^\vep \cdots x_{n_{a_{k-1}+i_{k}}}^\varepsilon x_{n_{a_{k-1}+i_{k}+1}}\cdots x_{n_{a_k}}x_{n_{a_{k}+1}}^\varepsilon\cdots x_{n_{a_k}+i_{k+1}}^\varepsilon.
\end{align}
Evidently, if $i_1=\cdots=i_k=0$, then $\alpha$ is an $\NN$-composition and $F_\alpha$ is a fundamental quasi-symmetric function of $\qsym$.

\subsection{Relationship between $M_\alpha$ and $F_\alpha$}
In this subsection, we generalize the well-known relationship between the monomial and fundamental quasi-symmetric functions, that is, $F_{\alpha}=\sum_{\beta\preceq\alpha}M_{\beta}$, to WC quasi-symmetric functions. First we extend the refining order $\preceq$ on the set $\mathcal{C}(\NN,n)$ of $\NN$-compositions to that of $\tilde{\NN}$-compositions.

Let
$\alpha=(\varepsilon^{i_1},\alpha_1,\cdots,\varepsilon^{i_k},\alpha_k,\varepsilon^{i_{k+1}})$,
$\beta=(\varepsilon^{j_1},\beta_1,\cdots,\varepsilon^{j_k},\beta_k,\varepsilon^{j_{k+1}})$ be two $\tilde{\NN}$-compositions of $n$,
where $i_1,\cdots,i_{k+1}$, $j_1,\cdots,j_{k+1}$ are nonnegative integers such that either $i_{k+1}=j_{k+1}=0$ or $i_{k+1}\geq1$, $j_{k+1}\geq1$, and $\alpha_1,\cdots,\alpha_k$, $\beta_1,\cdots,\beta_k$ are
$\NN$-compositions such that $|\alpha_l|=|\beta_l|$ for $l=1,\cdots,k$.
We extended the order $\preceq$ on $\mathcal{C}(\NN,n)$ to the set $\mathcal{C}(\tilde{\NN},n)$  of all $\tilde{\NN}$-compositions of $n$, still denoted by $\preceq$,
by setting $\alpha\preceq\beta$
if $i_1\leq j_1, \alpha_1\preceq \beta_1, \cdots, i_k\leq j_k, \alpha_k\preceq\beta_k, i_{k+1}\leq j_{k+1}$.
For example, $(1,2,\varepsilon^2,1,3,2,\varepsilon)\preceq(3,\varepsilon^2,1,\varepsilon,5,\varepsilon^3)$, but
$(1,2,\varepsilon^2,1,3,2)$ and $(3,\varepsilon^2,1,\varepsilon,5,\varepsilon^3)$ are not comparable.

\begin{prop}\label{fundamentalbasmonomialbs}
Let $\alpha=(\varepsilon^{i_1},\alpha_1,\cdots,\varepsilon^{i_k},\alpha_k,\varepsilon^{i_{k+1}})$ be an $\tilde\NN$-composition,
where $i_1,\cdots,i_{k+1}$ are nonnegative integers, $\alpha_1,\cdots,\alpha_k$ are $\NN$-compositions,
then
\begin{align}\label{fmalphas}
F_{\alpha}= \sum_{\beta\preceq\alpha}c_{\alpha,\beta} M_\beta
\end{align}
where $c_{\alpha,\beta}=\binom{i_1}{j_1}\cdots
\binom{i_k}{j_k}\binom{i_{k+1}-1}{j_{k+1}-1}$  if
$\beta=(\varepsilon^{j_1},\beta_1,\cdots,\varepsilon^{j_k},\beta_k,\varepsilon^{j_{k+1}})$ such that $\beta\preceq\alpha$, with the convention that $\binom{-1}{-1}=1$.
In particular, $c_{\alpha,\alpha}=1$.
Moreover, when $\alpha$ runs through all $\tilde\NN$-compositions, the elements $F_\alpha$'s, with the notation $F_{\emptyset}=1$, form a
$\ZZ$-basis for $\Wcqsym$.
\end{prop}
\begin{proof}
Notice that restricting the map $\theta$ in Eq.~\eqref{thetabijcntowc} to the set $\tilde{\NN}\backslash\{0\}$, we obtain an isomorphism of additive monoids
from $\tilde{\NN}\backslash\{0\}$ to $\NN$. So following the proof of \cite[Proposition 4.3]{Ygz2016} by substituting $0$ with $\varepsilon$,
one has
\begin{align}\label{falbinommbet}
F_{\alpha}=\sum_{0\leq j_r\leq i_r,\ 1\leq r\leq k}\binom{i_1}{j_1}\cdots\binom{i_k}{j_k}
\sum_{(\alpha_{p1},\cdots,\alpha_{pr_p})\preceq \alpha_p,\ 1\leq p\leq k}\ M'_{(\varepsilon^{j_1},
\alpha_{11},\cdots,\alpha_{1r_1},\cdots,\varepsilon^{j_k},\alpha_{k1},\cdots,\alpha_{kr_k},\varepsilon^{i_{k+1}})}
\end{align}
where
\begin{align*}
M'_{(\varepsilon^{j_1},\alpha_{11},\cdots,\alpha_{1r_1},\cdots,\varepsilon^{j_k},\alpha_{k1},\cdots,\alpha_{kr_k},\varepsilon^{i_{k+1}})}
=\sum x_{n_1}^{\varepsilon}\cdots x_{n_{j_1}}^{\varepsilon}x_{n_{j_1+1}}^{\alpha_{11}}\cdots x_{n_{j_1+r_1}}^{\alpha_{1r_1}}\cdots
x_{n_{j_1+\cdots+j_{k-1}+r_1+\cdots+r_{k-1}+1}}^{\varepsilon}\cdots\\
x_{n_{j_1+\cdots+j_{k}+r_1+\cdots+r_{k-1}}}^{\varepsilon}x_{n_{j_1+\cdots+j_{k}+r_1+\cdots+r_{k-1}+1}}^{\alpha_{k1}}\cdots x_{n_{j_1+\cdots+j_{k}+r_1+\cdots+r_{k}}}^{\alpha_{kr_k}}
x_{n_{j_1+\cdots+j_{k}+r_1+\cdots+r_{k}+1}}^{\varepsilon}\cdots x_{n_{j_1+\cdots+j_{k}+r_1+\cdots+r_{k}+i_{k+1}}}^{\varepsilon}
\end{align*}
and the summation is subject to the condition
$$
n_1<n_2<\cdots<n_{j_1+\cdots+j_{k}+r_1+\cdots+r_{k}} <n_{j_1+\cdots+j_{k}+r_1+\cdots+r_{k}+1}\leq n_{j_1+\cdots+j_{k}+r_1+\cdots+r_{k}+2 } \cdots\leq n_{j_1+\cdots+j_{k}+r_1+\cdots+r_{k}+i_{k+1}}.
$$
By the definition of WC monomial quasi-symmetric functions in Eq.~\eqref{eq:quasimf}, we have
\begin{align*}
M'_{(\varepsilon^{j_1},\alpha_{11},\cdots,\alpha_{1r_1},\cdots,\varepsilon^{j_k},\alpha_{k1},\cdots,\alpha_{kr_k},\varepsilon^{i_{k+1}})}
=\sum_{j_{k+1}=1}^{i_{k+1}}\binom{i_{k+1}-1}{j_{k+1}-1}
M_{(\varepsilon^{j_1},\alpha_{11},\cdots,\alpha_{1r_1},\cdots,\varepsilon^{j_k},\alpha_{k1},\cdots,\alpha_{kr_k},\varepsilon^{j_{k+1}})},
\end{align*}
which together with Eq.~\eqref{falbinommbet} yields Eq.~\eqref{fmalphas}.

According to the definition of $\preceq$, there exist only finitely many $\tilde{\NN}$-compositions less than  a given $\tilde{\NN}$-composition.
It follows from Eq.~\eqref{fmalphas} that
the transition matrix which expresses the $F_{\alpha}$ in terms of the $M_{\beta}$,
with respect to any linear order of $\tilde{\NN}$-compositions compatible with $\preceq$, is an upper triangular with $1$ on the main diagonal.
Hence, it is an invertible matrix since $\QQ$ is contained in $\bfk$, which shows that
$F_{\alpha}$ is a basis for $\Wcqsym$. In fact, it is
a $\ZZ$-basis since the diagonal entries are actually $1$'s, not merely nonzero.
\end{proof} \vspace{3mm}

Next we give the explicit transformation formula
expressing $M_\alpha$ in terms of the $F_\beta$'s.

\begin{prop}\label{fundamentalbastomonomialbs}
Adopt the notation $c_{\alpha,\beta} $ given in Proposition \ref{fundamentalbasmonomialbs}. For any $\tilde{\NN}$-composition $\alpha$, we have
\begin{align}\label{fphasmaleq}
M_\alpha=\sum_{\beta\preceq\alpha} (-1)^{\ell(\beta)-\ell(\alpha)}c_{\alpha,\beta} F_\beta.
\end{align}
\end{prop}
For instance, if $\alpha=(\varepsilon,2,\varepsilon^3)$, then

\begin{align*}
F_{\alpha}=&M_{(\varepsilon,2,\varepsilon^3)}+2M_{(\varepsilon,2,\varepsilon^2)}+M_{(\varepsilon,2,\varepsilon)}
+M_{(2,\varepsilon^3)}+2M_{(2,\varepsilon^2)}+M_{(2,\varepsilon)}\cr
&+M_{(\varepsilon,1,1,\varepsilon^3)}+2M_{(\varepsilon,1,1,\varepsilon^2)}+M_{(\varepsilon,1,1,\varepsilon)}
+M_{(1,1,\varepsilon^3)}+2M_{(1,1,\varepsilon^2)}+M_{(1,1,\varepsilon)},
\end{align*}
and
\begin{align*}
M_{\alpha}=&F_{(\varepsilon,2,\varepsilon^3)}-2F_{(\varepsilon,2,\varepsilon^2)}+F_{(\varepsilon,2,\varepsilon)}
-F_{(2,\varepsilon^3)}+2F_{(2,\varepsilon^2)}-F_{(2,\varepsilon)}\cr
&-F_{(\varepsilon,1,1,\varepsilon^3)}+2F_{(\varepsilon,1,1,\varepsilon^2)}-F_{(\varepsilon,1,1,\varepsilon)}
+F_{(1,1,\varepsilon^3)}-2F_{(1,1,\varepsilon^2)}+F_{(1,1,\varepsilon)}.
\end{align*}

\begin{proof}
Let $A=(a_{\alpha,\beta})$ be the upper triangular matrix which expresses the $F_{\alpha}$ in terms of the $M_{\beta}$'s,
with respect to any linear order of $\tilde{\NN}$-compositions compatible with $\preceq$. Then, in view of Proposition \ref{fundamentalbasmonomialbs},
for any $\alpha,\beta\in \mathcal{C}(\tilde{\NN},n)$, we have
$a_{\alpha,\beta}=0$ if $\beta\npreceq\alpha$, and
$a_{\alpha,\beta}=c_{\alpha,\beta}$ if $\beta\preceq\alpha$.
So it suffices to show that the inverse matrix of $A$ is $B=(b_{\alpha,\beta})$, where $b_{\alpha,\beta}=0$ if $\beta\npreceq\alpha$, and
$b_{\alpha,\beta}=(-1)^{\ell(\beta)-\ell(\alpha)}c_{\alpha,\beta}$ if $\beta\preceq\alpha$ for all $\alpha,\beta\in \mathcal{C}(\tilde{\NN},n)$.
In other words, we only need to show that for any given $\alpha,\gamma\in \mathcal{C}(\tilde{\NN},n)$ with $\gamma\preceq\alpha$, we have
$\sum_{\gamma\preceq\beta\preceq\alpha}a_{\alpha,\beta}b_{\beta,\gamma}=\delta_{\alpha,\gamma}$, that is,
\begin{align*}
\sum_{\gamma\preceq\beta\preceq\alpha} (-1)^{\ell(\gamma)-\ell{(\beta)}}c_{\alpha,\beta}c_{\beta,\gamma}=\delta_{\alpha,\gamma},
\end{align*}
where $\delta_{\alpha,\gamma}$  denotes the Kronecker delta.

Let $\alpha=(\varepsilon^{i_1},\alpha_1,\cdots,\varepsilon^{i_k},\alpha_k,\varepsilon^{i_{k+1}})$,
$\gamma=(\varepsilon^{j_1},\gamma_1,\cdots,\varepsilon^{j_k},\gamma_k,\varepsilon^{j_{k+1}})$ be the given $\tilde{\NN}$-compositions of $n$ with $\gamma\preceq\alpha$,
then $i_{k+1}=0$ if and only if $j_{k+1}=0$. Let
$\beta=(\varepsilon^{t_1},\beta_1,\cdots,\varepsilon^{t_k},\beta_k,\varepsilon^{t_{k+1}})$
be an $\tilde{\NN}$-composition such that $\gamma\preceq\beta\preceq\alpha$. Then, by the definition of the partial order $\preceq$, we know that
$\gamma\preceq\beta\preceq\alpha$ is equivalent to $\bar{\gamma}\preceq\bar{\beta}\preceq\bar{\alpha}$ and
$j_p\leq t_p\leq i_p$ for $p=1,2,\cdots,k+1$. Thus, we have
\begin{align*}
&\sum_{\gamma\preceq\beta\preceq\alpha} (-1)^{\ell(\gamma)-\ell{(\beta)}}c_{\alpha,\beta}c_{\beta,\gamma}\\
=&(-1)^{\ell(\gamma)}\sum_{\gamma\preceq\beta\preceq\alpha} (-1)^{\ell{(\beta)}}
\binom{i_1}{t_1}\cdots\binom{i_k}{t_k}\binom{i_{k+1}-1}{t_{k+1}-1}
\binom{t_1}{j_1}\cdots\binom{t_k}{j_k}\binom{t_{k+1}-1}{j_{k+1}-1}\\
=&(-1)^{\ell(\gamma)}\left(\sum_{\bar{\gamma}\preceq\bar{\beta}\preceq\bar{\alpha}} (-1)^{\ell{(\bar{\beta})}}\right)
\left(\prod_{p=1}^{k}\sum_{t_p=j_p}^{ i_p}(-1)^{t_p} \binom{i_p}{t_p}\binom{t_p}{j_p}\right)
\left(\sum_{t_{k+1}=j_{k+1}}^{i_{k+1}}(-1)^{t_{k+1}}
\binom{i_{k+1}-1}{t_{k+1}-1}\binom{t_{k+1}-1}{j_{k+1}-1}\right).
\end{align*}
Since $\bar{\gamma}\preceq\bar{\beta}\preceq\bar{\alpha}$ are $\NN$-compositions, we have
$\sum_{\bar{\gamma}\preceq\bar{\beta}\preceq\bar{\alpha}} (-1)^{\ell{(\bar{\beta})}}=(-1)^{\ell{(\bar{\alpha})}}\delta_{\bar{\alpha}\,\bar{\gamma}}$.
For each $p=1,2,\cdots, k$, we have
\begin{align*}
\sum_{t_p=j_p}^{i_p}(-1)^{t_p} \binom{i_p}{t_p}\binom{t_p}{j_p}
=\sum_{t_p=j_p}^{i_p}(-1)^{t_p} \binom{i_p}{j_p}\binom{i_p-j_p}{t_p-j_p}
=(-1)^{j_p} \binom{i_p}{j_p}\sum_{t_p=j_p}^{i_p}(-1)^{t_p-j_p} \binom{i_p-j_p}{t_p-j_p}=(-1)^{i_p}\delta_{i_pj_p}.
\end{align*}
An analogous argument shows that
\begin{align*}
\sum_{t_{k+1}=j_{k+1}}^{i_{k+1}}(-1)^{t_{k+1}}
\binom{i_{k+1}-1}{t_{k+1}-1}\binom{t_{k+1}-1}{j_{k+1}-1}=(-1)^{i_{k+1}}\delta_{i_{k+1}j_{k+1}}.
\end{align*}
Therefore,
\begin{align*}
\sum_{\gamma\preceq\beta\preceq\alpha} (-1)^{\ell(\gamma)-\ell{(\beta)}}c_{\alpha,\beta}c_{\beta,\gamma}
=&(-1)^{\ell(\gamma)}\left((-1)^{\ell{(\bar{\alpha})}}\delta_{\bar{\alpha}\,\bar{\gamma}}\right)
\left(\prod_{p=1}^k(-1)^{i_p}\delta_{i_pj_p}\right)\left((-1)^{i_{k+1}}\delta_{i_{k+1}j_{k+1}}\right)\\
=&(-1)^{\ell(\gamma)+\ell(\alpha)}\delta_{\alpha,\gamma}=\delta_{\alpha,\gamma},
\end{align*}
completing the proof.
\end{proof}

\subsection{The antipode of $\Wcqsym$}

It useful to give a formula for the coefficient of $M_{\beta}$ in $S_{W}(M_{\alpha})$ for $ \alpha^r\leq\beta$. Here we use
the subscript $W$ to indicate that $S_W=S_{\tilde{\NN}}$ is the antipode of $\Wcqsym$.

\begin{lemma}\label{lemcoeffmbainsmalph}
Let $\alpha$ and $\beta$ be $\tilde{\NN}$-compositions such that $\alpha^r\leq \beta$.
\begin{enumerate}
\item\label{lemcoeffmbainsmalpha} There exists a uniquely $\NN$-composition $L\models \ell(\bar{\alpha})$ such that $\bar{\beta}=L\circ\bar{\alpha}^r$;
\item\label{lemcoeffmbainsmalphb} Let $\alpha=(\varepsilon^{i_1},\alpha_1,\cdots,\varepsilon^{i_k},\alpha_k,\varepsilon^{i_{k+1}})$,
$\beta=(\varepsilon^{j_1},\beta_1,\cdots,\varepsilon^{j_p},\beta_p,\varepsilon^{j_{p+1}})$ where $i_1,\cdots,i_k,i_{k+1},j_1,$ $\cdots,$ $j_p, j_{p+1}\in\NN$,
$\alpha_1,\cdots,\alpha_k,\beta_1,\cdots,\beta_p\in\PP$,
and let $L=(l_1,l_2,\cdots,l_p)$ be such that $\bar\beta=L\circ{\bar\alpha}^r$.
Then the coefficient of $M_{\beta}$ in  $S_{W}(M_{\alpha})$ is
\begin{align*}
(-1)^{\ell(\alpha)}\binom{i_{b_1}}{j_1}\binom{i_{b_{2}}+1}{j_{2}+1}\cdots\binom{i_{b_{p}}+1}{j_{p}+1}\binom{i_{b_{p+1}}}{j_{p+1}}
\end{align*}
where $b_t=l_p+l_{p-1}+\cdots+l_{t}+1$ for $1\leq t\leq p$, and $b_{p+1}=1$.
\end{enumerate}
\end{lemma}
\begin{proof}
\eqref{lemcoeffmbainsmalpha}
Let $\alpha=(\varepsilon^{i_1},\alpha_1,\cdots,\varepsilon^{i_k},\alpha_k,\varepsilon^{i_{k+1}})$ and
$\beta=(\varepsilon^{j_1},\beta_1,\cdots,\varepsilon^{j_p},\beta_p,\varepsilon^{j_{p+1}})$ be $\tilde{\NN}$-compositions satisfying $\alpha^r\leq \beta$,
where $i_1,\cdots,i_k,i_{k+1},j_1,$ $\cdots,$ $j_p, j_{p+1}\in\NN$,
$\alpha_1,\cdots,\alpha_k,\beta_1,\cdots,\beta_p\in\PP$.
Then there exist positive integers $l_1,l_2,\cdots,l_{p-1}$ such that $\beta_1=\alpha_k+\cdots+\alpha_{k-{l_1}+1}$,
$\beta_2=\alpha_{k-{l_1}}+\cdots+\alpha_{k-{l_1}-l_2+1}$, $\cdots$, $\beta_p=\alpha_{k-{l_1}-l_2-\cdots-l_{p-1}}+\cdots+\alpha_{1}$. Take
 $l_p=k-(l_1+l_2+\cdots+l_{p-1})$ and
$L=(l_1,l_2,\cdots,l_p)$. So we have $L\models \ell(\bar{\alpha})$ and
$$\bar{\beta}=(\beta_1,\beta_2,\cdots,\beta_p)=L\circ(\alpha_k,\alpha_{k-1},\cdots,\alpha_1)=L\circ\bar{\alpha}^r.$$
Since $\bar{\beta}$ and $\bar{\alpha}^r$ are $\NN$-composition, the uniqueness of $L$ is obviously.

\eqref{lemcoeffmbainsmalphb}
The idea of the proof is to count the number of ways of obtaining $\beta$ from
$\alpha^r$ such that $\alpha^r\leq\beta$.
Let $b_t=l_p+l_{p-1}+\cdots+l_{t}+1$ where $1\leq t\leq p$, and let $b_{p+1}=1$.
Then  we can obtain the $\tilde\NN$-composition
$$
\beta'=:(\varepsilon^{i_{b_1}},\beta_1,\varepsilon^{i_{b_{2}}},\beta_2,\cdots,\varepsilon^{i_{b_{p}}},\beta_p,\varepsilon^{i_{b_{p+1}}})
$$
in a unique way by summing all the entries of the sub-composition
$$
(\alpha_{b_{t}-1},\varepsilon^{i_{b_{t}-1}},\alpha_{b_{t}-2},\cdots, \varepsilon^{i_{b_{t+1}+1}},\alpha_{b_{t+1}})
$$ of $\alpha^r$ to give the entry $\beta_t$ for $1\leq t\leq p$,
and leaving all other entries of $\alpha$ unchanged. Now we can obtain $\beta$ from $\beta'$ by considering the expression
$$
(\overbrace{\varepsilon \star\cdots\star\varepsilon}^{i_{b_1}\ {\rm factors}}\star\beta_1\star
\overbrace{\varepsilon \star\cdots\star\varepsilon}^{i_{b_{2}}\ {\rm factors}}\star\beta_2\star\cdots\star
\overbrace{\varepsilon \star\cdots\star\varepsilon}^{i_{b_p}\ {\rm factors}}\star\beta_p\star
\overbrace{\varepsilon \star\cdots\star\varepsilon}^{{i_{b_{p+1}}}\ {\rm factors}}),
$$
and replacing each $\star$ by either a plus sign or a comma in
$$
\binom{i_{b_1}}{j_1}\binom{i_{b_{2}}+1}{j_{2}+1}\cdots\binom{i_{b_{p}}+1}{j_{p}+1}\binom{i_{b_{p+1}}}{j_{p+1}}
$$
ways. Then we are done by Proposition \ref{antieqsym}.
\end{proof}

As an example, we present the antipodes
of $M_{(\varepsilon,1,\varepsilon,2)}$ and $M_{\varepsilon^n}$.

\begin{exam}\label{antipexameqsym}
Let $\alpha=(\varepsilon_1,1,\varepsilon_2,2)$. Here subscripts are added to the two occurrences of $\varepsilon$ in $\alpha$ for easy identification. Then
\begin{align*}
S_W(M_{\alpha})=&(-1)^{4}\left[\right.M_{(2,\varepsilon_2,1,\varepsilon_1)}+M_{(2+\varepsilon_2,1,\varepsilon_1)}+M_{(2,\varepsilon_2+1,\varepsilon_1)}
\cr
&+M_{(2,\varepsilon_2,1+\varepsilon_1)}+M_{(2+\varepsilon_2+1,\varepsilon_1)}+M_{(2+\varepsilon_2,1+\varepsilon_1)}+M_{(2,\varepsilon_2+1+\varepsilon_1)}
+M_{(2+\varepsilon_2+1+\varepsilon_1)}\left.\right]\cr
=&M_{(2,\varepsilon,1,\varepsilon)}+2M_{(2,1,\varepsilon)}+M_{(2,\varepsilon,1)}+M_{(3,\varepsilon)}+2M_{(2,1)}+M_{3}.
\end{align*}
For the $\tilde\NN$-composition $\varepsilon^n$ we have
\begin{align*}
S_W(M_{\varepsilon^n})=(-1)^n\sum_{i=0}^{n-1}\binom{n-1}{i}M_{\varepsilon^{i+1}}.
\end{align*}
\end{exam}

\subsection{Surjective Hopf homomorphism from $\Wcqsym$ to $\qsym$}\label{surhom}
In this subsection, we will show that $\qsym$ is a Hopf subalgebra
and a quotient Hopf algebra of $\Wcqsym$.
Let $\mathcal{C}_{\varepsilon}=\{(\varepsilon,\delta)|\delta\in \mathcal{C}(\tilde{\NN})\}$, and let $\mathcal{C}_{N}= \mathcal{C}(\tilde{\NN}) \backslash \mathcal{C}_{\varepsilon}$.
Then $\mathcal{C}_{N}$ consists of the empty composition and all $\tilde\NN$-compositions  whose first entry is a positive integer.
Define a linear map
\begin{align}\label{eqhahvarphi}
\varphi: \Wcqsym \rightarrow \qsym,   M_{\alpha}\mapsto \begin{cases}
(-1)^{\ell_{\varepsilon}(\alpha)}M_{\bar{\alpha}},& \alpha\in \mathcal{C}_{N},\\
0,&\alpha\in \mathcal{C}_{\varepsilon}.
\end{cases}
\end{align}
It is easy to see that the restriction of $\varphi$ to $\qsym$ is the identity map. Thus $\varphi$ is surjective and $\Wcqsym$ is the direct sum of $\Ker\varphi$ and $\qsym$.
The rest of this subsection is devoted to showing that the map $\varphi$ is a Hopf algebra homomorphism and determining the Hopf ideal $\ker \varphi$.

\begin{lemma}\label{lemmaprovealghom}
Let $a$ be a positive integer, and let $\alpha$, $\beta$ be $\tilde{\NN}$-compositions. Then
\begin{align*}
\varphi(M_{(a,\alpha*\beta)})=(-1)^{\ell_{\varepsilon}(\alpha)+\ell_{\varepsilon}(\beta)}M_{(a,\bar{\alpha}*\bar{\beta})}.
\end{align*}
\end{lemma}
\begin{proof}
We proceed by induction on $\ell(\alpha)+\ell(\beta)\geq0$.
If $\ell{(\alpha)+\ell(\beta)}=0$, then $\alpha$ and $\beta$ are empty and so $\varphi(M_{a})=M_{a}$ is automatic.
For each $n\geq 1$, assume that the desired equation holds for all $\alpha,\beta\in\mathcal{C}(\tilde{\NN})$ with $\ell{(\alpha)+\ell(\beta)}< n$.
Now take any $\alpha,\beta\in\mathcal{C}(\tilde{\NN})$ such that $\ell{(\alpha)+\ell(\beta)}=n$.

Since there is nothing to prove if either $\alpha$ or $\beta$ is empty,
we can assume that $\alpha=(\alpha_1,\delta)$ and $\beta=(\beta_1,\sigma)$ for some  $\alpha_1,\beta_1\in \tilde{\NN}\backslash\{0\}$ and
$\delta,\sigma\in\mathcal{C}(\tilde{\NN})$.
Thus,
\begin{align}\label{eqvarmaquasishualbd}
\varphi(M_{(a,\alpha*\beta)})=\varphi(M_{(a,\alpha_1,\delta*\beta)})+\varphi(M_{(a,\beta_1,\alpha*\sigma)})+\varphi(M_{(a,\alpha_1+\beta_1,\delta*\sigma)}).
\end{align}

If $\alpha_1=\varepsilon$, then the induction hypothesis gives
\begin{align*}
\varphi(M_{(a,\alpha_1,\delta*\beta)})
=-\varphi(M_{(a,\delta*\beta)})
=-(-1)^{\ell_{\varepsilon}(\delta)+\ell_{\varepsilon}(\beta)}M_{(a,\bar{\delta}*\bar{\beta})}
=(-1)^{\ell_{\varepsilon}(\alpha)+\ell_{\varepsilon}(\beta)}M_{(a,\bar{\alpha}*\bar{\beta})}.
\end{align*}
If $\alpha_1\neq\varepsilon$, then, by the induction hypothesis,
\begin{align*}
\varphi(M_{(\alpha_1,\delta*\beta)})=(-1)^{\ell_{\varepsilon}(\delta)+\ell_{\varepsilon}(\beta)}M_{(\alpha_1,\bar{\delta}*\bar{\beta})}
=(-1)^{\ell_{\varepsilon}(\alpha)+\ell_{\varepsilon}(\beta)}M_{(\alpha_1,\bar{\delta}*\bar{\beta})}.
\end{align*}
Since $a$ is a positive integer, it is straightforward to see that
\begin{align*}
\varphi(M_{(a,\alpha_1,\delta*\beta)})
=(-1)^{\ell_{\varepsilon}(\alpha)+\ell_{\varepsilon}(\beta)}M_{(a,\alpha_1,\bar{\delta}*\bar{\beta})}
\end{align*}
holds. To summarize, we have proved
\begin{align}\label{eqvaphmaal1}
\varphi(M_{(a,\alpha_1,\delta*\beta)})=\begin{cases}
(-1)^{\ell_{\varepsilon}(\alpha)+\ell_{\varepsilon}(\beta)}M_{(a,\bar{\alpha}*\bar{\beta})},& \alpha_1=\varepsilon,\\
(-1)^{\ell_{\varepsilon}(\alpha)+\ell_{\varepsilon}(\beta)}M_{(a,\alpha_1,\bar{\delta}*\bar{\beta})},&\alpha_1\neq\varepsilon.
\end{cases}
\end{align}
In an analogous manner we obtain
\begin{align}\label{eqvaphmabe2}
\varphi(M_{(a,\beta_1,\alpha*\sigma)})=\begin{cases}
(-1)^{\ell_{\varepsilon}(\alpha)+\ell_{\varepsilon}(\beta)}M_{(a,\bar{\alpha}*\bar{\beta})},& \beta_1=\varepsilon,\\
(-1)^{\ell_{\varepsilon}(\alpha)+\ell_{\varepsilon}(\beta)}M_{(a,\beta_1,\bar{\alpha}*\bar{\sigma})},&\beta_1\neq\varepsilon
\end{cases}
\end{align}
and
\begin{align}\label{eqvaphmaalbe3}
\varphi(M_{(a,\alpha_1+\beta_1,\delta*\sigma)})
=\begin{cases}
(-1)^{\ell_{\varepsilon}(\alpha)+\ell_{\varepsilon}(\beta)-1}M_{(a,\bar{\alpha}*\bar{\beta})},& \alpha_1=\beta_1=\varepsilon,\\
(-1)^{\ell_{\varepsilon}(\alpha)+\ell_{\varepsilon}(\beta)-1}M_{(a,\beta_1,\bar{\alpha}*\bar{\sigma})},& \alpha_1=\varepsilon, \beta_1\neq\varepsilon, \\
(-1)^{\ell_{\varepsilon}(\alpha)+\ell_{\varepsilon}(\beta)-1}M_{(a,\alpha_1,\bar{\delta}*\bar{\beta})},& \alpha_1\neq\varepsilon, \beta_1=\varepsilon,\\
(-1)^{\ell_{\varepsilon}(\alpha)+\ell_{\varepsilon}(\beta)}M_{(a,\alpha_1+\beta_1,\bar{\delta}*\bar{\sigma})},& \alpha_1\neq\varepsilon, \beta_1\neq\varepsilon.
\end{cases}
\end{align}
Combining Eqs.\eqref{eqvarmaquasishualbd}, \eqref{eqvaphmaal1}, \eqref{eqvaphmabe2} and \eqref{eqvaphmaalbe3}, we have
\begin{align*}
\varphi(M_{(a,\alpha*\beta)})
=\begin{cases}
(-1)^{\ell_{\varepsilon}(\alpha)+\ell_{\varepsilon}(\beta)}
\left(M_{(a,\alpha_1,\bar{\delta}*\bar{\beta})}+M_{(a,\beta_1,\bar{\alpha}*\bar{\sigma})}+M_{(a,\alpha_1+\beta_1,\bar{\delta}*\bar{\sigma})}\right),
& \alpha_1\neq\varepsilon, \beta_1\neq\varepsilon,\\
(-1)^{\ell_{\varepsilon}(\alpha)+\ell_{\varepsilon}(\beta)}M_{(a,\bar{\alpha}*\bar{\beta})},& {\rm otherwise}.
\end{cases}
\end{align*}
Since  $\bar{\alpha}=(\alpha_1,\bar{\delta})$ and $\bar{\beta}=(\beta_1,\bar{\sigma})$ if $\alpha_1\neq\varepsilon$, $\beta_1\neq\varepsilon$, it follows that
\begin{align*}
M_{(a,\bar{\alpha}*\bar{\beta})}
=M_{(a,\alpha_1,\bar{\delta}*\bar{\beta})}+M_{(a,\beta_1,\bar{\alpha}*\bar{\sigma})}+M_{(a,\alpha_1+\beta_1,\bar{\delta}*\bar{\sigma})}
\end{align*}
so that
\begin{align*}
\varphi(M_{(a,\alpha*\beta)})=(-1)^{\ell_{\varepsilon}(\alpha)+\ell_{\varepsilon}(\beta)}M_{(a,\bar{\alpha}*\bar{\beta})},
\end{align*}
as required.
\end{proof}

\begin{lemma}\label{lemmaalghom}
The linear map $\varphi$ is an algebra homomorphism.
\end{lemma}
\begin{proof}
As $\varphi$ is a linear map, we only need to prove $\varphi(M_{\alpha}M_{\beta})=\varphi(M_{\alpha})\varphi(M_{\beta})$
for all $\tilde{\NN}$-compositions $\alpha$ and $\beta$. Since the result is immediate if $\alpha$ or $\beta$ is empty, we can write
$\alpha=(a,\delta)$ and $\beta=(b,\sigma)$ for some $a,b\in \tilde{\NN}\backslash\{0\}$, $\delta=(\delta_1,\delta_2,\cdots,\delta_m)\in\mathcal{C}(\tilde{\NN})$ and
$\sigma=(\sigma_1,\sigma_2,\cdots,\sigma_n)\in\mathcal{C}(\tilde{\NN})$ where $m,n\in \NN$.

If $a=b=\varepsilon$, then it is straightforward to see that $\varphi(M_{\alpha}M_{\beta})=0=\varphi(M_{\alpha})\varphi(M_{\beta})$.
If there is exactly one of $a,b$ is $\varepsilon$, then, without loss of generality, we may assume that $a\neq\varepsilon$, $b=\varepsilon$.
Then
\begin{align*}
\varphi(M_{\alpha}M_{\beta})=\varphi(M_{(a,\delta*\beta)}+M_{(b,\alpha*\sigma)}+M_{(a+b,\delta*\sigma)})
=\varphi(M_{(a,\delta*\beta)})+\varphi(M_{(a,\delta*\sigma)}),
\end{align*}
and hence, by Lemma \ref{lemmaprovealghom},
\begin{align*}
\varphi(M_{\alpha}M_{\beta})=&(-1)^{\ell_{\varepsilon}(\delta)+\ell_{\varepsilon}(\beta)}M_{(a,\bar{\delta}*\bar{\beta})}
+(-1)^{\ell_{\varepsilon}(\delta)+\ell_{\varepsilon}(\sigma)}M_{(a,\bar{\delta}*\bar{\sigma})}
=0.
\end{align*}
On the other hand, $\varphi(M_{\alpha})\varphi(M_{\beta})=0$ clearly holds, which means that $\varphi(M_{\alpha}M_{\beta})=\varphi(M_{\alpha})\varphi(M_{\beta})$.

It remains to consider the case when neither $a$ nor $b$ is equal to $\varepsilon$. Since
\begin{align*}
\alpha*\beta=(a,\delta*\beta)+(b,\alpha*\sigma)+(a+b,\delta*\sigma),
\end{align*}
one has
\begin{align*}
\varphi(M_{\alpha}M_{\beta})=\varphi(M_{(a,\delta*\beta)})+\varphi(M_{(b,\alpha*\sigma)})+\varphi(M_{(a+b,\delta*\sigma)}),
\end{align*}
which together with Lemma \ref{lemmaprovealghom} yields
\begin{align*}
\varphi(M_{\alpha}M_{\beta})=&(-1)^{\ell_{\varepsilon}(\delta)+\ell_{\varepsilon}(\beta)}M_{(a,\bar{\delta}*\bar{\beta})})
+(-1)^{\ell_{\varepsilon}(\alpha)+\ell_{\varepsilon}(\sigma)}M_{(b,\bar{\alpha}*\bar{\sigma})}
+(-1)^{\ell_{\varepsilon}(\delta)+\ell_{\varepsilon}(\sigma)}M_{(a+b,\bar{\delta}*\bar{\sigma})}\\
=&(-1)^{\ell_{\varepsilon}(\alpha)+\ell_{\varepsilon}(\beta)}\left(M_{(a,\bar{\delta}*\bar{\beta})}
+M_{(b,\bar{\alpha}*\bar{\sigma})}
+M_{(a+b,\bar{\delta}*\bar{\sigma})}\right).
\end{align*}
The hypothesis on  $a$ and $b$ implies that $\bar{\alpha}=(a,\bar{\delta})$ and $\bar{\beta}=(b,\bar{\sigma})$. Hence
\begin{align*}
\varphi(M_{\alpha}M_{\beta})
=(-1)^{\ell_{\varepsilon}(\alpha)+\ell_{\varepsilon}(\beta)}M_{\bar{\alpha}}M_{\bar{\beta}}
=\varphi(M_{\alpha})\varphi(M_{\beta}),
\end{align*}
which completes the proof.
\end{proof}

\begin{lemma}\label{lemmacoalghometq}
The linear map $\varphi$ is a coalgebra homomorphism.
\end{lemma}
\begin{proof}
Take any $\tilde{\NN}$-composition $\alpha=(\alpha_1,\alpha_2,\cdots,\alpha_n)$. If $\alpha_1=\varepsilon$, then, by
Eq.~\eqref{eqhahvarphi}, $\Delta_{\NN} \varphi(M_\alpha)=0$. According to Eq.~\eqref{mqsydeltcop1},
\begin{align*}
(\varphi\otimes \varphi )\Delta_W(M_\alpha)
&=(\varphi\otimes \varphi )\left(\sum_{k=0}^nM_{(\alpha_1,\cdots,\alpha_k)}\otimes M_{(\alpha_{k+1},\cdots,\alpha_n)}\right)
=\sum_{k=0}^n\varphi\left(M_{(\alpha_1,\cdots,\alpha_k)}\right)
\otimes \varphi\left(M_{(\alpha_{k+1},\cdots,\alpha_n)}\right)=0.
\end{align*}
Thus, we have $\Delta_{\NN}\, \varphi(M_\alpha)=(\varphi\otimes \varphi )\Delta_W (M_\alpha)$ for any weak composition $\alpha$ with $\alpha_1=\varepsilon$.

If $\alpha_1\neq\varepsilon$, we may write $\alpha$ as
$(\alpha_1,\varepsilon^{i_1},\alpha_2,\varepsilon^{i_2},\cdots,\alpha_n,\varepsilon^{i_n})$
where $i_j\in\NN$ and $\alpha_j\in\PP$ for all $j=1,\cdots,n$. It follows from Eqs.~\eqref{mqsydeltcop1} and \eqref{eqhahvarphi} that
\begin{align*}
(\varphi\otimes \varphi )\Delta_W (M_\alpha)
&=\sum_{\alpha=\beta\cdot\gamma}\varphi(M_{\beta})\otimes \varphi(M_{\gamma})\\
&=\sum_{j=0}^n\varphi\left(M_{(\alpha_1,\varepsilon^{i_1},\cdots,\alpha_j,\varepsilon^{i_j})}\right)\otimes
\varphi\left(M_{(\alpha_{j+1},\varepsilon^{i_{j+1}},\cdots,\alpha_n,\varepsilon^{i_n})}\right)\\
&=\sum_{j=0}^n(-1)^{i_1+\cdots+i_j}M_{(\alpha_1,\cdots,\alpha_j)}\otimes
(-1)^{i_{j+1}+\cdots+i_n} M_{(\alpha_{j+1},\cdots,\alpha_n)}\\
&=(-1)^{\ell_{\varepsilon}(\alpha)}\sum_{j=0}^nM_{(\alpha_1,\cdots,\alpha_j)}\otimes M_{(\alpha_{j+1},\cdots,\alpha_n)}\\
&=(-1)^{\ell_{\varepsilon}(\alpha)} \Delta_{\NN}\left(M_{\bar{\alpha}}\right)\\
&=\Delta_{\NN}\,\varphi\left(M_{\alpha}\right).
\end{align*}
Hence $\Delta_{\NN}\,\varphi\left(M_{\alpha}\right)=(\varphi\otimes \varphi )\Delta_W (M_\alpha)$ also holds for all weak composition $\alpha$ with
 $\alpha_1\neq\varepsilon$.

Moreover, from Eq.\eqref{mqsydeltcop2} we obtain
$$
\epsilon_{\NN}\,\varphi(M_{\alpha})=\delta_{\alpha,\emptyset}=\epsilon_W(M_{\alpha}).
$$
Thus $\epsilon_{\NN}\,\varphi=\epsilon_W$. The proof is completed.
\end{proof}

\begin{theorem}\label{thmhopfalghomantipode}
The linear map $\varphi:\Wcqsym \rightarrow \qsym$ is a surjective Hopf algebra  homomorphism.
\end{theorem}
\begin{proof}
As noted after the definition of $\varphi$, the map is surjective.
Together with Lemmas \ref{lemmaalghom} and \ref{lemmacoalghometq}, it suffices to show that the antipodes are compatible, that is,
$\varphi(S_W(M_\alpha))=S_{\NN}(\varphi(M_\alpha))$ for all $\tilde{\NN}$-compositions $\alpha$.
We divide our proof in the two cases of $\alpha\in \mathcal{C}_\varepsilon$ and $\alpha\in \mathcal{C}_N$.
If $\alpha\in \mathcal{C}_\varepsilon$, then we may write $\alpha=(\varepsilon,\delta)$ for some $\tilde{\NN}$-composition $\delta$. Then by the definition of $\varphi$, $S_{\NN}(\varphi(M_\alpha))=S_{\NN}(0)=0$. On the other hand, by Proposition \ref{antieqsym},
\begin{align*}
\varphi(S_W(M_\alpha))
=&(-1)^{\ell(\alpha)}\sum_{J\models \ell(\alpha)}\varphi(M_{J\circ (\delta^r,\varepsilon)})\\
=&(-1)^{\ell(\alpha)}\left(\sum_{J\models \ell(\alpha)-1}\varphi(M_{(J\circ \delta^r,\varepsilon)})+\sum_{J\models \ell(\alpha)-1}\varphi (M_{J\circ \delta^r})\right)\\
=&(-1)^{\ell(\alpha)}\sum_{J\models \ell(\alpha)-1}\left(\varphi(M_{(J\circ \delta^r,\varepsilon)})+\varphi (M_{J\circ \delta^r})\right)\\
=&0.
\end{align*}
Then $\varphi(S_W(M_\alpha))=S_{\NN}(\varphi(M_\alpha))$ for all $\alpha\in \mathcal{C}_\varepsilon$.

If $\alpha\in \mathcal{C}_N$, then we may write $\alpha=(\alpha_1,\varepsilon^{i_2},\alpha_2,\cdots,\varepsilon^{i_k},\alpha_k,\varepsilon^{i_{k+1}})$
where $i_j\in\NN$ for $j=2,\cdots,k+1$ and $\alpha_j\in\PP$ for all $j=1,2,\cdots,k$. It follows from Proposition \ref{antieqsym} that
\begin{align}\label{eqsqvarphima}
S_{\NN}(\varphi(M_\alpha))=&S_{\NN}\left((-1)^{\ell_{\varepsilon}(\alpha)}M_{\bar{\alpha}}\right)
=(-1)^{\ell_{\varepsilon}(\alpha)}(-1)^{\ell{(\bar{\alpha}})}\sum_{J\models k}M_{J\circ {\bar{\alpha}^r}}
=(-1)^{\ell(\alpha)}\sum_{\bar{\alpha}^r\leq\delta}M_{\delta}.
\end{align}
On the other hand,
with the notation in Lemma \ref{lemcoeffmbainsmalph}, one has
\begin{align*}
S_W(M_\alpha)
=(-1)^{\ell(\alpha)}\sum_{L\models k}
\binom{i_{b_1}}{j_1}\binom{i_{b_{2}}+1}{j_{2}+1}\cdots\binom{i_{b_{p}}+1}{j_{p}+1}\binom{i_{b_{p+1}}}{j_{p+1}}M_{\beta},
\end{align*}
where $\beta$ is of the form $(\varepsilon^{j_1},\beta_1,\cdots,\varepsilon^{j_p},\beta_p,\varepsilon^{j_{p+1}})$.
Notice that $i_{b_{p+1}}=i_1=0$ so that $j_{p+1}=0$. By Eq.~\eqref{eqhahvarphi}, $\varphi(M_\beta)=0$ if $j_{1}\neq0$, so we have
\begin{equation}\label{eqmbarbeta=mbet}
\begin{split}
\varphi (S_W(M_\alpha))
=&(-1)^{\ell(\alpha)}\sum_{\bar{\beta}=L\circ\bar{\alpha}^r}\binom{i_{b_{2}}+1}{j_{2}+1}\cdots\binom{i_{b_{p}}+1}{j_{p}+1}\varphi(M_{\beta})\\
=&(-1)^{\ell(\alpha)}\sum_{\bar{\alpha}^r\leq \bar{\beta}}(-1)^{\ell_\varepsilon(\beta)}\binom{i_{b_{2}}+1}{j_{2}+1}\cdots\binom{i_{b_{p}}+1}{j_{p}+1} M_{\bar{\beta}}
\end{split}
\end{equation}
Thus, for each $\NN$-composition $\delta\geq \bar{\alpha}^r$,
the coefficient of $M_\delta$ on the right hand side of Eq.\eqref{eqmbarbeta=mbet} is
\begin{align*}
(-1)^{\ell(\alpha)}\sum_{{0\leq j_{t}\leq i_{b_t}},\ {2\leq t\leq p}}
(-1)^{\ell_{\varepsilon}(\beta)}\binom{i_{b_{2}}+1}{j_{2}+1}\cdots\binom{i_{b_{p}}+1}{j_{p}+1}
=&(-1)^{\ell(\alpha)}\prod_{t=2}^p\sum_{j_{t}=0}^{i_{b_t}}(-1)^{j_{t}}\binom{i_{b_t}+1}{j_{t}+1}\\
=&(-1)^{\ell(\alpha)}\prod_{t=2}^p\left(-\sum_{j_t=0}^{i_{b_t}+1}(-1)^{j_t}\binom{i_{b_t}+1}{j_t}+1\right)\\
=&(-1)^{\ell(\alpha)}.
\end{align*}
Therefore, by comparing Eqs.~\eqref{eqsqvarphima} and \eqref{eqmbarbeta=mbet}, $\varphi (S_W(M_\alpha))=S_{\NN}(\varphi(M_\alpha))$ holds for all  $\alpha\in \mathcal{C}_N$.
\end{proof}

We now give a basis for the free $\bfk$-module $\Ker \varphi$ and a Hopf ideal generating set for the ideal $\Ker \varphi$.
Let $B_{\varepsilon}:=\{M_{\alpha}|\alpha\in \mathcal{C}_{\varepsilon}\}$, and let
$B_{N}:=\{M_{\alpha}+(-1)^{\ell_{\varepsilon}(\alpha)+1}M_{\bar{\alpha}}|\alpha\in \mathcal{C}_{N}\backslash\mathcal{C}(\NN)\}$.

\begin{theorem}\label{thm:basisforKervarphi}
The disjoint union $B_{\varepsilon}\uplus B_{N}$
 a basis for the free $\bfk$-module $\Ker\varphi$.
\end{theorem}
\begin{proof}
It is obvious that $B_\varepsilon$ and $B_{N}$ are disjoint sets.
Notice that $\alpha\in\mathcal{C}_{N}\backslash{\mathcal{C}(\NN)}$ implies that $\bar{\alpha}\in \mathcal{C}(\NN)$. So if
\begin{align*}
\sum_{\alpha\in \mathcal{C}_{\varepsilon}}c_{\alpha}M_{\alpha}
+\sum_{\alpha\in\mathcal{C}_{N}\backslash{\mathcal{C}(\NN)}}c_{\alpha}\left(M_{\alpha}+(-1)^{\ell_{\varepsilon}(\alpha)+1}M_{\bar{\alpha}}\right)=0,
\end{align*}
then we must have $c_{\alpha}=0$ for all $\alpha\in \mathcal{C}_\varepsilon\cup \left(\mathcal{C}_N\backslash{\mathcal{C}(\NN)}\right)$,
and therefore $B_{\varepsilon}\uplus B_{N}$ is linearly independent.

By Eq.~\eqref{eqhahvarphi}, $B_\varepsilon\uplus B_{N}$ is contained in $\Ker\varphi$.
Let $f\in \Ker \varphi$. Since $\{M_{\alpha}|\alpha\in \mathcal{C}(\tilde{\NN})\}$ is a basis of $\Wcqsym$, there is a unique linear combination
$$
f=\sum_{\alpha\in \mathcal{C}(\tilde{\NN})}c_{\alpha}M_{\alpha}, \quad c_\alpha\in \bfk.
$$
Noticing that $\mathcal{C}(\tilde{\NN})=\mathcal{C}_{\varepsilon}\uplus \mathcal{C}_{N}$, we have $f=g+h$ where
\begin{align}\label{eq:f=g+hgh}
g=\sum_{\alpha\in \mathcal{C}_{\varepsilon}}c_{\alpha}M_{\alpha}, \quad h=\sum_{\alpha\in \mathcal{C}_{N}}c_{\alpha}M_{\alpha}.
\end{align}
Since $B_{\varepsilon}=\{M_{\alpha}|\alpha\in \mathcal{C}_{\varepsilon}\}\subseteq\Ker\varphi$, we have $g\in \Ker\varphi$ so that  $h\in \Ker \varphi$.
Writing
\begin{align}\label{eq:h-beta=cm}
h_{\beta}=\sum_{\alpha\in \mathcal{C}_{N}\ {\rm with}\ \beta=\bar{\alpha}}c_{\alpha}M_{\alpha}
\end{align}
for each $\beta\in \calc(\NN)$, we conclude that $h=\sum_{\beta\in \calc(\NN)}h_{\beta}$.
Then
\begin{align*}
\sum_{\beta\in \calc(\NN)}\left(\sum_{\alpha\in \mathcal{C}_{N}\ {\rm with}\ \beta=\bar{\alpha}}(-1)^{\ell_{\varepsilon}(\alpha)}c_{\alpha}\right)M_{\beta}
=\sum_{\beta\in \calc(\NN)}\varphi(h_{\beta})=\varphi(h)=0.
\end{align*}
Because $\{M_\beta\,|\,\beta\in\calc(\NN)\}$ is a basis for $\qsym$, one has
$$
\sum_{\alpha\in \mathcal{C}_{N}\ {\rm with}\ \beta=\bar{\alpha}}(-1)^{\ell_{\varepsilon}(\alpha)}c_{\alpha}=0
$$
for each $\beta\in\calc(\NN)$.
Therefore, for each $\beta\in\calc(\NN)$, we have
$$
h_{\beta}=h_{\beta}-\left(\sum_{\alpha\in \mathcal{C}_{N}\ {\rm with}\ \beta=\bar{\alpha}}(-1)^{\ell_{\varepsilon}(\alpha)}c_{\alpha}\right)M_{\beta}\\
=\sum_{\alpha\in \mathcal{C}_{N}\ {\rm with}\ \beta=\bar{\alpha}}c_{\alpha}\left(M_{\alpha}+(-1)^{\ell_{\varepsilon}(\alpha)+1}M_{\beta}\right).
$$
Notice that $M_{\alpha}+(-1)^{\ell_{\varepsilon}(\alpha)+1}M_{\beta}=0$ for all $\NN$-compositions $\alpha\in \mathcal{C}(\NN)$ with $\beta=\bar{\alpha}$,
so
\begin{align}\label{eq:hbeta=cImB-N}
h_\beta=\sum_{\alpha\in \mathcal{C}_{N}\backslash{\mathcal{C}(\NN)}\ {\rm with}\ \beta=\bar{\alpha}}c_{\alpha}\left(M_{\alpha}+(-1)^{\ell_{\varepsilon}(\alpha)+1}M_{\beta}\right).
\end{align}
Combining Eqs.~\eqref{eq:f=g+hgh}, \eqref{eq:h-beta=cm} and \eqref{eq:hbeta=cImB-N}, we obtain
\begin{align*}
f=\sum_{\alpha\in \mathcal{C}_{\varepsilon}}c_{\alpha}M_{\alpha}
+\sum_{\alpha\in\mathcal{C}_{N}\backslash{\mathcal{C}(\NN)}}c_{\alpha}\left(M_{\alpha}+(-1)^{\ell_{\varepsilon}(\alpha)+1}M_{\bar{\alpha}}\right).
\end{align*}
Thus, we can write $f$ as a linear combination of $B_{\varepsilon}\uplus B_{N}$,
completing the proof.
\end{proof}

\begin{lemma}\label{lemma:mr+meprninI}
Let $\alpha=(a,\alpha_1,\cdots,\alpha_r,\varepsilon,\alpha_{r+2},\cdots,\alpha_n)\in \mathcal{C}_{N}$  where $r\in\NN$ and $a,\alpha_1,\cdots,\alpha_r\in\PP$.
If $I$ is the ideal of $\Wcqsym$  generated by $B_{\varepsilon}$, then
\begin{align*}
M_{\alpha}
+M_{(a,\alpha_1,\cdots,\alpha_r,\alpha_{r+2},\cdots,\alpha_n)}\in I.
\end{align*}
\end{lemma}
\begin{proof}
We prove by induction on $r\geq 0$.
If $r=0$, then $\alpha=(a,\varepsilon,\alpha_2,\cdots,\alpha_n)$.
It follows from $M_{(\varepsilon,\alpha_2,\cdots,\alpha_n)}\in B_\varepsilon\subseteq I$ that
\begin{align*}
M_{\alpha}+M_{(a,\alpha_2,\cdots,\alpha_n)}+M_{(\varepsilon,a*(\alpha_2,\cdots,\alpha_n))}=M_aM_{(\varepsilon,\alpha_2,\cdots,\alpha_n)}\in I,
\end{align*}
which together with $M_{(\varepsilon,a*(\alpha_2,\cdots,\alpha_n))}\in B_\varepsilon\subseteq I$ yields that
the assertion is true.

Now for a given $m\geq 1$, assume that the claim holds for $r$ less than $m$. We prove by induction on
$k$, $1\leq k \leq m+1$, that
\begin{align}\label{eqlem:mep1tok+malinI}
M_{(a,\alpha_1,\cdots,\alpha_{k-1},(\alpha_{k},\cdots,\alpha_m)*(\varepsilon,\alpha_{m+2},\cdots,\alpha_n))}
+M_{(a,\alpha_1,\cdots,\alpha_{k-1},(\alpha_{k},\cdots,\alpha_m)*(\alpha_{m+2},\cdots,\alpha_n))}\in I,
\end{align}
from which the proof follows by setting $k=m+1$.

Since $M_{(\varepsilon,\alpha_{m+2},\cdots,\alpha_n)}\in B_\varepsilon$, we have
\begin{align*}
M_{(a,\alpha_1,\cdots,\alpha_m)} M_{(\varepsilon,\alpha_{m+2},\cdots,\alpha_n)}
=M_{(a,(\alpha_1,\cdots,\alpha_m)*(\varepsilon,\alpha_{m+2},\cdots,\alpha_n))}
+M_{(\varepsilon,(a,\alpha_1,\cdots,\alpha_m)*(\alpha_{m+2},\cdots,\alpha_n))}
+M_{(a,(\alpha_1,\cdots,\alpha_m)*(\alpha_{m+2},\cdots,\alpha_n))}
\end{align*}
is in $I$, which together with $M_{(\varepsilon,(a,\alpha_1,\cdots,\alpha_m)*(\alpha_{m+2},\cdots,\alpha_n))}\in B_\varepsilon\subseteq I$ implies that
\begin{align*}
M_{(a,(\alpha_1,\cdots,\alpha_m)*(\varepsilon,\alpha_{m+2},\cdots,\alpha_n))}
+M_{(a,(\alpha_1,\cdots,\alpha_m)*(\alpha_{m+2},\cdots,\alpha_n))}\in I.
\end{align*}
Thus, Eq.~\eqref{eqlem:mep1tok+malinI} holds for $k=1$.

We next suppose $k>1$  and assume that Eq.~\eqref{eqlem:mep1tok+malinI} holds for $k-1$. Then the induction hypothesis on $k$ gives
\begin{equation}\label{eq:malep6m}
\begin{split}
&M_{(a,\alpha_1,\cdots,\alpha_{k-2},(\alpha_{k-1},\cdots,\alpha_m)*(\varepsilon,\alpha_{m+2},\cdots,\alpha_n))}
+M_{(a,\alpha_1,\cdots,\alpha_{k-2},(\alpha_{k-1},\cdots,\alpha_m)*(\alpha_{m+2},\cdots,\alpha_n))}\\
=&M_{(a,\alpha_1,\cdots,\alpha_{k-1},(\alpha_k,\cdots,\alpha_m)*(\varepsilon,\alpha_{m+2},\cdots,\alpha_n))}
+M_{(a,\alpha_1,\cdots,\alpha_{k-2},\varepsilon,(\alpha_{k-1},\cdots,\alpha_m)*(\alpha_{m+2},\cdots,\alpha_n))}\\
&+M_{(a,\alpha_1,\cdots,\alpha_{k-1},(\alpha_{k},\cdots,\alpha_m)*(\alpha_{m+2},\cdots,\alpha_n))}
+M_{(a,\alpha_1,\cdots,\alpha_{k-2},(\alpha_{k-1},\cdots,\alpha_m)*(\alpha_{m+2},\cdots,\alpha_n))}\in I.
\end{split}
\end{equation}
If we write $(\alpha_{k-1},\cdots,\alpha_m)*(\alpha_{m+2},\cdots,\alpha_n)=\sum_{\delta\in\mathcal{C}(\tilde{\NN})}c_\delta\delta$ with $c_\delta\in \bfk$, then
\begin{equation}\label{eq:ma+mep=sumcm}
\begin{split}
M_{(a,\alpha_1,\cdots,\alpha_{k-2},\varepsilon,(\alpha_{k-1},\cdots,\alpha_m)*(\alpha_{m+2},\cdots,\alpha_n))}
&+M_{(a,\alpha_1,\cdots,\alpha_{k-2},(\alpha_{k-1},\cdots,\alpha_m)*(\alpha_{m+2},\cdots,\alpha_n))}\\
&=\sum_{\delta\in\mathcal{C}(\tilde{\NN})}c_\delta \left(M_{(a,\alpha_1,\cdots,\alpha_{k-2},\varepsilon,\delta)}+M_{(a,\alpha_1,\cdots,\alpha_{k-2},\delta)}\right).
\end{split}
\end{equation}
Notice that $k\leq m+1$, so $k-2\leq m-1$. By the induction hypothesis on $r$, we obtain that $M_{(a,\alpha_1,\cdots,\alpha_{k-2},\varepsilon,\delta)}+M_{(a,\alpha_1,\cdots,\alpha_{k-2},\delta)}\in I$.
Therefore, Eqs.~\eqref{eq:malep6m} and \eqref{eq:ma+mep=sumcm} imply Eq.~\eqref{eqlem:mep1tok+malinI} holds. Take $k=m+1$ gives the desired result, completing the proof.
\end{proof}

\begin{lemma}\label{lemma:ma-1mbiI}
Let $I$ be the ideal of $\Wcqsym$  generated by $B_{\varepsilon}$.
For each $\tilde{\NN}$-composition $\alpha\in \mathcal{C}_{N}$, we have
\begin{align}\label{eqmal-1epmalinI}
M_{\alpha}+(-1)^{\ell_{\varepsilon}(\alpha)+1}M_{\bar{\alpha}}\in I.
\end{align}
\end{lemma}
\begin{proof}
The proof is by induction on $\ell_{\varepsilon}(\alpha)$.
If $\ell_{\varepsilon}(\alpha)=0$, then  $\alpha=\bar{\alpha}$, so the assertion is true.
Now assume that Eq.~\eqref{eqmal-1epmalinI} holds for all $\alpha\in \mathcal{C}_{N}$ with $\ell_{\varepsilon}(\alpha)=s-1$ for some $s\in\PP$, and take $\alpha\in \mathcal{C}_{N}$
such that $\ell_{\varepsilon}(\alpha)=s$.

Since $\alpha\in \mathcal{C}_{N}$, we can write $\alpha=(a,\alpha_1,\cdots,\alpha_n)$ for some $a\in\PP$ and $(\alpha_1,\cdots,\alpha_n)\in \mathcal{C}(\tilde{\NN})$ with
$n\in\NN$. It follows from $\ell_{\varepsilon}(\alpha)=s\in\PP$ that there exists a positive integer $r$ such that $\alpha_1,\alpha_2,\cdots,\alpha_r$ are positive integers, and $\alpha_{r+1}$ is equal to $\varepsilon$. Then $\alpha=(a,\alpha_1,\cdots,\alpha_r,\varepsilon,\alpha_{r+2},\cdots,\alpha_n)$.

Let $\sigma=(a,\alpha_1,\cdots,\alpha_r,\alpha_{r+2},\cdots,\alpha_n)$. Then $\bar\alpha=\bar\sigma$ and $M_{\alpha}+M_{\sigma}\in I$
in view of Lemma \ref{lemma:mr+meprninI}.
Notice that  $\ell_{\varepsilon}(\sigma)=\ell_{\varepsilon}{(\alpha)}-1=s-1$, so, by the induction hypothesis,
$$
M_{\sigma}+(-1)^{\ell_{\varepsilon}(\sigma)+1}M_{\bar{\sigma}}\in I.
$$
Thus,
$$
M_{\alpha}+(-1)^{\ell_{\varepsilon}(\alpha)+1}M_{\bar{\alpha}}=\left(M_{\alpha}+M_{\sigma}\right)-\left(M_{\sigma}+(-1)^{\ell_{\varepsilon}(\sigma)+1}M_{\bar{\sigma}}\right)\in I,
$$
as required.
\end{proof}

In view of Theorem \ref{thm:basisforKervarphi} and Lemma \ref{lemma:ma-1mbiI}, we have
\begin{coro}\label{cor:Hopfidealvarphietoq}
The Hopf ideal $\Ker\varphi$ is the ideal generated by the set $B_{\vep}$.
\end{coro}

\section{The connection between Rota-Baxter algebras and \qsym}
\mlabel{sec:CBRBAQSYM}

We are going to establish the connection between the free Rota-Baxter algebras and the algebras of quasi-symmetric functions.
Firstly, we briefly recall the construction of a free Rota-Baxter algebra by the mixable shuffle product.
See \cite{Gub,G-K1} for more details.

\begin{defn}\label{RBAlambda}
For a fixed $\lambda\in \bfk$, a {\bf Rota-Baxter $\mathbf{k}$-algebra (RBA)
of weight $\lambda$} is a commutative $\mathbf{k}$-algebra $R$ together with a $\mathbf{k}$-linear
operator $P: R\rightarrow R$ that satisfies the Rota-Baxter identity
\begin{equation}\label{rtequ}
P(x)P(y)=P(xP(y))+P(P(x)y)+\lambda P(xy)\qquad \forall x,y\in R.
\end{equation}
Such an operator $P$ is called a {\bf Rota-Baxter operator (RBO) of weight $\lambda$}. If $R$ is only assumed to be a non-unitary
$\mathbf{k}$-algebra, we call $R$ a non-unitary Rota-Baxter $\mathbf{k}$-algebra of weight $\lambda$.
\end{defn}

A morphism $f: (R, P)\rightarrow (S, Q)$ of Rota-Baxter $\mathbf k$-algebras is a $\mathbf k$-algebra homomorphism $f: R\rightarrow S$ such that
$f(P(a)) = Q(f(a))$ for all $a\in R$.
Given a commutative $\mathbf {k}$-algebra $A$ that is not necessarily unitary,
the free commutative Rota-Baxter $\mathbf {k}$-algebra on $A$ is defined to be a \rbto
$\mathbf {k}$-algebra $(F(A),P_A)$ together with a $\mathbf{k}$-algebra homomorphism
$j_A: A\rightarrow F(A)$ with the property that, for any \rbto $\mathbf{k}$-algebra $(R, P)$ and any
$\mathbf{k}$-algebra homomorphism $f: A\rightarrow R$, there is a unique Rota-Baxter $\mathbf{k}$-algebra homomorphism
$\tilde{f}: (F(A),P_A)\rightarrow (R, P)$ such that $f=\tilde{f}\circ j_A$ as $\mathbf{k}$-algebra homomorphisms.

We recall the free commutative Rota-Baxter $\mathbf {k}$-algebra on $A$ constructed in~\cite{G-K1}  by using the mixable shuffle algebra.

Let $A$ be a commutative $\bfk$-algebra that is not necessarily unitary. For a given $\lambda\in\bfk$, the {\bf mixable shuffle algebra of weight $\lambda$
generated by $A$} is the $\bfk$-module
\begin{align*}
\sha_{\bfk}^{+}(A):=\sha_{\bfk,\lambda}^{+}(A)=\bigoplus_{k\geq 0} A^{\otimes k}=\mathbf{k}\oplus A\oplus A^{\otimes2}\oplus\cdots,
\quad \text{ where\ } A^{\otimes k}= \underbrace{A\otimes A\otimes\cdots \otimes A}_{ k-{\text factors}},
\end{align*}
equipped  with the mixable  shuffle product $*$ defined as follows.

For pure tensors $\fraka =a_1\otimes \cdots \otimes a_m\in A^{\otimes m}$ and $\frakb =b_1\otimes \cdots \otimes b_n\in A^{\otimes n}$,
a {\bf shuffle} of $\fraka $ and $\frakb $ is a tensor
list of $a_i$ and $b_j$ without change the natural orders of the $a_i$'s and the $b_j$'s.  More generally, for the given $\lambda\in\bfk$,
a {\bf mixable shuffle} (of weight $\lambda$) of $\fraka$ and $\frakb$ is a shuffle of $\fraka$ and $\frakb$ where some  of the pairs $a_i\ot b_j$ are replaced by $\lambda a_ib_j$. The {\bf mixable shuffle product} $\fraka *_{\lambda}\frakb$ of $\fraka$ and $\frakb$ is the sum of all mixable shuffles.
For example
$$ a_1 *_{\lambda}(b_1\ot b_2)=a_1\ot b_1\ot b_2 + b_1\ot a_1\ot b_2+b_1\ot b_2\ot a_1+\lambda(a_1b_1)\ot b_2+b_1\ot \lambda(a_1b_2),$$
where $\lambda(a_1b_1)\ot b_2$ comes from $a_1\ot b_1\ot b_2$ by ``mixing" or merging $a_1\ot b_1$ and $b_1\ot \lambda(a_1b_2)$ comes from $b_1\ot a_1\ot b_2$ by ``mixing" or merging $a_1\ot b_2$. The last shuffle $b_1\ot b_2\ot a_1$ does not yield any mixed term since $a_1$ is not before any $b_j$, $j=1, 2$.

The mixable shuffle product can be also defined by the quasi-shuffle product given by the recursion~\cite{EG2006,Ho}
\begin{equation}\label{mixableshuprod}
\fraka* \frakb =a_1\otimes((a_2\otimes \cdots \otimes a_m) * \frakb)+b_1\otimes(\fraka *  (b_2\otimes \cdots \otimes b_n))
+\lambda(a_1b_1)\otimes((a_2\otimes \cdots \otimes a_m)*  (b_2\otimes \cdots \otimes b_n))
\end{equation}
with the convention that $1*  \fraka =\fraka *  1=\fraka $.
The mixable shuffle product equipped $\sha_{\bfk}^{+}(A)$ with a commutative algebra structure \cite{G-K1}.

Let $A$ be a commutative unitary $\bfk$-algebra. Define
\begin{align*}
\sha_{\bfk,\lambda}(A):=A\otimes \sha_{\bfk}^{+}(A)=\bigoplus_{k\geq1} A^{\otimes k}
\end{align*}
to be the tensor product of the algebras $A$ and $\sha_{\bfk}^{+}(A)$. For notational convenience we will write $a\otimes \mathbf{1}_{\bfk}=a$ for $a\in A$.
For two pure tensors $a_0\otimes \fraka =a_0\otimes a_1\otimes \cdots \otimes a_m$
 and $b_0\otimes \frakb =b_0\otimes b_1\otimes \cdots \otimes b_n$, the {\bf augmented mixable shuffle product}
 $\diamond_{\lambda}$ on $\sha_{\bfk,\lambda}(A)$ is defined by
\begin{align}\label{defdiamond}
(a_0\otimes \fraka )\diamond_{\lambda} (b_0\otimes \frakb ):=\begin{cases}
a_0b_0,&m=n=0,\\
(a_0b_0)\otimes \fraka, & m>0,n=0, \\
(a_0b_0)\otimes \frakb, & m=0,n>0,\\
(a_0b_0)\otimes(\fraka \ast   \frakb ),&m>0,n>0.
\end{cases}
\end{align}
Thus, we have the algebra isomorphism (embedding of the second tensor factor)
\begin{align*}
\eta:(\sha_{\bfk}^{+}(A),*)\rightarrow (\mathbf{1}_A\otimes \sha_{\bfk}^{+}(A), \diamond_{\lambda})
\end{align*}
The pair of products $*$ and $\diamond_{\lambda}$ is a special case of the double products~\cite{Gub} in Rota-Baxter algebras.

The following theorem was established in \cite{G-K1}.
\begin{theorem}\label{freerbamsp}
The algebra $(\sha_{\bfk,\lambda}(A),\diamond_\lambda)$, with the linear
operator $P_A :\sha_{\bfk}(A)\rightarrow \sha_{\bfk}(A)$ sending $\fraka$ to $1\otimes \fraka$, is a free commutative \rbto algebra
of weight $\lambda$ generated by $A$.
\end{theorem}

For the rest of the paper, we will assume that $\lambda=\mathbf{1}_{\bfk}$ and drop $\lambda$ from the notation.

The free commutative Rota-Baxter algebra on one generator $x$ is the same as the free commutative Rota-Baxter algebra on
the polynomial algebra $\bfk[x]$.
So the special case of Theorem \ref{freerbamsp} in this case can be restated as
\begin{theorem}\label{rbaonegenx}
The $\bfk$-module
\begin{align}\label{shaklambdax11}
\sha(x):=\sha_{\bfk}\left(\bfk[x]\right)=\bfk[x]\otimes \sha_{\bfk}^{+}(\bfk[x])=\bigoplus_{k\geq 1}  \bfk[x]^{\otimes k},
\end{align}
with the product in Eq.~\eqref{defdiamond} and the operator $P_{\bfk[x]}: \fraka\mapsto 1\ot \fraka$, is the free unitary Rota--Baxter algebra on $x$.
\end{theorem}

Since a $\bfk$-linear basis of $\bfk[x]^{\otimes k}$ is the set
\begin{align*}
\{x^{\alpha_1} \otimes\cdots\otimes x^{\alpha_k}|\alpha_i\geq0, 1\leq i\leq k\},
\end{align*}
by denoting $x^{\otimes \alpha}:=x^{\alpha_1}\ot \cdots \ot x^{\alpha_k}$ for $\alpha=(\alpha_1,\cdots,\alpha_k)\in \WC(\NN)$,
we have
\begin{align*}
\sha(x):=\bigoplus_{\alpha\in \NN^k, k\geq 1}\bfk x^{\otimes \alpha}.
\end{align*}
Moreover, we also have
\begin{align*}
\sha^{+}(x):=\sha_{\bfk}^{+}(x)=\bigoplus_{\alpha\in \WC(\NN)}\bfk x^{\otimes \alpha},
\end{align*}
where $x^{\otimes\emptyset}=\mathbf{1}_{\bfk}$ is the identity of $\bfk$.

Through the bijection $\theta:\calc(\tilde{\NN}) \to \WC(\NN)$ defined in Eq.~\eqref{thetabijcntowc}, we obtain a natural linear bijection
\begin{align}\label{eq:rhoshatowc}
\rho: \bfk\mathcal{C}(\tilde{\NN})\to \sha^{+}(x), \quad \alpha\mapsto x^{\otimes\theta(\alpha)},
\end{align}
which indeed is an algebra isomorphism from the quasi-shuffle algebra $\bfk\mathcal{C}(\tilde{\NN})$ to
the mixable shuffle algebra $\sha^{+}(x)$.

We are now in a position to realize the free {\em unitary} Rota-Baxter algebra $\sha(x)$ as a subalgebra of a semigroup exponent formal power series algebra.
The authors of \cite{Ygz2016} realized the free commutative nonunitary Rota-Baxter of weight $1$ on one generator in terms of left weak composition quasi-symmetric
functions. This will be generalized for the  unitary case in what follows.

For $X:=\{x_n\,|\, n\geq 1\}$ with its natural ordering, denote $\bar{X}:=\{x_0\}\cup X =\{x_n\,|\, n\geq 0\}$ with the linear order such that
$x_0<x_1$. Identifying $\bfk[[X]]_{\tilde{\NN}}$ as a subalgebra of $\bfk[[\bar{X}]]_{\tilde{\NN}}$, we form the subalgebra
${\bar\Wcqsym}(\bar{X})$ (or ${\bar\Wcqsym}$ for short) of $\bfk[[\bar{X}]]_{\tilde{\NN}}$ by
\begin{align}\label{x0k[x0]otimes}
{\bar\Wcqsym}:=x_0^\varepsilon\mathbf{k}[x_0]\, \Wcqsym \cong x_0^\varepsilon\mathbf{k}[x_0]\ot \Wcqsym,
\end{align}
where $\Wcqsym$ is the algebra of WC quasi-symmetric functions with $\{M_{\alpha}\,|\,\alpha\in \mathcal{C}(\tilde{\NN})\}$ as a linear basis. Therefore, we have
\begin{align}\label{barqsymkbasis}
{\bar\Wcqsym}=\bigoplus_{(\alpha_0,\cdots,\alpha_k)\in \mathcal{C}(\tilde{\NN})\backslash\{\emptyset\}} \bfk {\bar M}_{(\alpha_0,\cdots,\alpha_k)},
\end{align}
where ${\bar M}_{(\alpha_0,\mathbf{\alpha})}:=x_0^{\alpha_0} M_{\mathbf{\alpha}}$.
For example,
$$
{\bar M}_{(\varepsilon,2,\varepsilon)}=x_0^{\varepsilon}\sum_{i_1<i_2}x_{i_1}^2x_{i_2}^{\varepsilon} \qquad {\rm and} \qquad
{\bar M}_{(2,\varepsilon,3)}=x_0^{2}\sum_{i_1<i_2}x_{i_1}^{\varepsilon}x_{i_2}^3.
$$
Note the special role played by the first entry $\alpha_0$.

Following the multiplication of WC quasi-symmetric functions in Proposition~\mref{prop:kcstoqsyms},  the product on ${\bar\Wcqsym}$ is given by
\begin{align}\label{productshapseetym}
{\bar M}_{(\alpha_0,\mathbf{\alpha})}{\bar M}_{(\beta_0,\mathbf{\beta})}={\bar M}_{(\alpha_0+\beta_0,\mathbf{\alpha}*\mathbf{\beta})},\qquad
\alpha_0,\beta_0\in \tilde{\NN}\backslash\{0\},\ \mathbf{\alpha}, \mathbf{\beta}\in \mathcal{C}(\tilde{\NN}),
\end{align}
with the convention $\emptyset*\alpha=\alpha*\emptyset=\alpha$ and $(\alpha,\emptyset)=\alpha$ for all $\alpha\in \mathcal{C}(\tilde{\NN})$.
Define a linear endomorphism $\overline{P}$ on ${\bar\Wcqsym}$ by  assigning
\begin{align*}
\overline{P}({\bar M}_{(\alpha_0,\mathbf{\alpha})})={\bar M}_{(\varepsilon,\alpha_0,\mathbf{\alpha})}
\end{align*}
and extend linearly.

\begin{theorem}\label{freerbau1}
$({\bar\Wcqsym} ,\overline{P})$ is the free commutative  unitary \rbto algebra of weight $1$ generated by $x_0$.
\end{theorem}

\begin{proof}
By definition, ${\bar\Wcqsym} $ is a subalgebra of the $\tilde{\NN}$-exponent power series algebra $\bfk[[\bar{X}]]_{\tilde{\NN}}$.
For any $\tilde{\NN}$-composition $(\alpha_0,\alpha)$ with $\alpha_0\in\tilde\NN$,
${\bar M}_\varepsilon \bar{M}_{(\alpha_0,\alpha)}=x_0^\varepsilon x_0^{\alpha_0} M_{\mathbf{\alpha}}=\bar{M}_{(\alpha_0,\alpha)}$
so that $\bar M_{\varepsilon}$ is the identity of ${\bar\Wcqsym}$.

As free $\mathbf{k}$-modules, $\sha(x)$ has the $\mathbf{k}$-linear basis
\begin{align*}
\{x^{\alpha_0}\otimes x^{\alpha_1}\otimes\cdots\otimes x^{\alpha_n}|(\alpha_0,\alpha_1,\cdots,\alpha_n)\in\mathcal{WC}(\mathbb{N})\backslash\{\emptyset\}\},
\end{align*}
while ${\bar\Wcqsym} $ has the $\mathbf{k}$-linear basis
\begin{align*}
\{{\bar{M}}_{(\alpha_0,{\alpha_1},\cdots,{\alpha_n})}|(\alpha_0,{\alpha_1},\cdots,{\alpha_n})\in \mathcal{C}(\tilde{\NN})\backslash\{\emptyset\}\}.
\end{align*}
Therefore, the assignment
$\bar{\varphi}({\bar{M}}_{(\alpha_0,{\alpha_1},\cdots,{\alpha_n})}):=x^{\theta(\alpha_0)}\otimes x^{\theta(\alpha_1)}\otimes\cdots\otimes x^{\theta(\alpha_n)}$,
where $\theta$ is given by Eq.~\eqref{thetabijcntowc},
defines a linear bijection
$\bar{\varphi}:{\bar\Wcqsym} \to \sha(x).$
By Eqs.~\eqref{defdiamond} and \eqref{productshapseetym},
$\bar{\varphi}$ is an algebra isomorphism. Further $\bar{\varphi}\overline{P}= P_{\bfk[x]}\bar{\varphi}$ by definition. Thus $({\bar\Wcqsym},\overline{P})$ is a Rota-Baxter algebra isomorphic to $(\sha(x),P_{\bfk[x]})$. This completes the proof.
\end{proof}

We now try to understand $\sha(x)$ from the perspective of Hopf algebras. In \cite{EG2006}, the authors showed that for a Hoffman set $X=\cup_{n\geq1}X_n$,
the mixable shuffle algebra $\sha^{+}(A)$, as a subalgebra of $\sha(A)$, carries a Hopf algebra structure, where
$A=\bfk\{X\}$ is the free $\bfk$-module on $X$. If we put $X_n=\{x_n\}$ for all $n\geq 1$ and $[x_i,x_j]=x_{i+j}$, then it is straightforward to check that
the map $\phi$ defined by
\begin{align*}
\phi(x_{i_1}\otimes x_{i_2}\otimes \cdots\otimes  x_{i_k})=\sum_{1\leq n_1<n_2<\cdots<n_k}x_{n_1}^{i_1}x_{n_2}^{i_2}\cdots x_{n_k}^{i_k}
\end{align*}
is an isomorphism of $\sha^{+}(A)$ onto the algebra $\qsym$ of quasi-symmetric functions over $\bfk$.
We now generalize this relationship to establish the connection between $\qsym$ and the free commutative Rota-Baxter algebra $\sha(x)$ of weight $\mathbf{1}_{\bfk}$
on the generator element $x$.

\begin{theorem}\label{thm:qsymsubalgquotRBA}
The algebra $\Wcqsym$ of weak composition quasi-symmetric functions is isomorphic to the algebra $\sha^{+}(x)$ and thus to the subalgebra
$\mathbf{1}_{\bfk}\otimes\sha^{+}(x)$ of the free commutative Rota-Baxter algebra of weight $\mathbf{1}_{\bfk}$ generated by $x$.
Moreover, $\qsym$ is a subalgebra of $\sha_{\bfk}(x)$ and a quotient Hopf algebra of $\mathbf{1}_{\bfk}\otimes\sha^{+}(x)$.
\end{theorem}
\begin{proof}
By Eqs. \eqref{eq:rhoshatowc} and \eqref{thetabijcntowc},  we have the isomorphisms
\begin{align*}
\begin{array}{ccccc}
\psi:\Wcqsym &\rightarrow  &\sha^{+}(x)&\rightarrow &\mathbf{1}_{\bfk}\otimes\sha^{+}(x)\subseteq \sha(x),\cr
M_\alpha     &\mapsto & x^{\otimes \theta(\alpha)}    &\mapsto &1\otimes x^{\otimes \theta(\alpha)}
\end{array}
\end{align*}
of $\bfk$-algebras, where the map $\theta$ is given by Eq.~\eqref{thetabijcntowc}.
Then $\qsym$ is a subalgebra of $\Wcqsym$ so that it is a subalgebra of $\sha_{\bfk}(x)$.
It follows from Theorem \ref{thmhopfalghomantipode} that $\qsym$ is
a quotient Hopf algebra of $\mathbf{1}_{\bfk}\otimes\sha^{+}(x)$. This completes the proof.
\end{proof}

Recall that the polynomial algebra $\mathbf{k}[x]$ is a Hopf algebra
with $x$ primitive. The coproduct $\Delta_{\mathbf{k}[x]}$ is defined by
\begin{align}\label{eq:Deltakxxm}
\Delta_{\mathbf{k}[x]}(x^m)=\sum_{p=0}^m\binom{m}{p}x^p\otimes x^{m-p},
\end{align}
the counit is given by $\epsilon_{\bfk[x]}(x^m)=\delta_{0,m}$ and the antipode $S_{\mathbf{k}[x]}$ is determined by
\begin{align}\label{eq:smk[x]}
S_{\mathbf{k}[x]}(x^m)=(-x)^m,\qquad m\geq0.
\end{align}

Recall~\cite{Abe} that for two $\bfk$-Hopf algebras $A$ and $B$, the tensor product $A\otimes B$ is also a Hopf algebra with
multiplication defined by $(a\otimes b)(c\otimes d)=ac\otimes bd$ and
comultiplication given by
\begin{align}\label{gener7alabdcop}
\Delta_{A\otimes B}:A\otimes B\rightarrow(A\otimes B)\otimes(A\otimes B),
\quad \Delta_{A\otimes B}(a\otimes b)=\sum_{(a),(b)}(a_{(1)}\otimes b_{(1)})\otimes(a_{(2)}\otimes b_{(2)}),
\end{align}
with the Sweedler notation $\Delta(c)=\sum_{(c)}c_{(1)}\otimes c_{(2)}$.
The counit is given by
\begin{align}\label{gener7counit}
\epsilon_{A\otimes B}:A\otimes B\rightarrow \bfk\otimes \bfk:=\bfk,\quad \epsilon_{A\otimes B}(a\otimes b)=\epsilon_{A}(a)\epsilon_{B}(b),
\end{align}
and the antipode is defined by
\begin{align}\label{anti96shax}
S_{A\otimes B}:A\otimes B\rightarrow A\otimes B,\quad S_{A\otimes B}(a\otimes b)=S_{A}(a)\otimes S_{B}(b).
\end{align}

Since $\sha(x)$ is the tensor product of the algebras $\bfk[x]$ and $\sha^{+}(x)$,
$\sha(x)$ is a Hopf algebra. More precisely, we have the following statement which specializes to~\cite{AGKO} and extends the Hopf algebras in~\cite{EG2006}.

\begin{coro}\label{shaklbbialg}
The free commutative unitary \rbto algebra
$(\sha(x),\diamond,\mu,\Delta,\epsilon,S)$ of weight $1$ generated by $x$  is a Hopf algebra, where
\begin{enumerate}
\item\label{coro:pdiamo}
$\diamond$ is the augmented mixable shuffle product;

\item\label{coro:mup} $\mu:\bfk\rightarrow \sha(x), \quad k\mapsto k$;

\item\label{coro:codeltaph} $\Delta:\ \sha(x) \rightarrow\sha(x)\otimes\sha(x)$, for a sequence $(a,{\alpha_1},\cdots,{\alpha_k})$ of nonnegative integers,
\begin{align*}
\Delta\left(x^{\otimes(a,{\alpha_1},\cdots,{\alpha_k})}\right)
= \sum_{i=0}^k\sum_{p=0}^a
\binom{a}{p}\left(x^{\otimes(p,{\alpha_1},\cdots,
{\alpha_i})}\right)
\otimes
\left(x^{\otimes({a-p},{\alpha_{i+1}},\cdots,
{\alpha_k})}\right),
\end{align*}
with the convention that $x^{\alpha_1}\otimes \cdots \otimes x^{\alpha_i}=1$ if $i=0$ and $x^{\alpha_{i+1}}\otimes \cdots \otimes x^{\alpha_k}=1$
if $i=k$;

\item\label{coro:epsilon} $\epsilon: \sha(x)\rightarrow\bfk$ is a $\bfk$-linear map, where if $w= x^{\otimes(a,{\alpha_1},\cdots,{\alpha_k})}$, then
$\epsilon(w)=\delta_{1,w}$;

\item\label{coro:pantipodehp} The antipode $S$ is given by
\begin{align*}
S\left(x^{\otimes(a,{\alpha_{1}}, \cdots, {\alpha_k})}\right)
=(-1)^{a+k}x^a\otimes\sum _{(i_1,i_2,\cdots,i_r)\models k}
x^{\otimes(\alpha_k+\cdots+\alpha_{k-i_1+1},\cdots,\alpha_{i_r}+\cdots+\alpha_{1})}.
\end{align*}
\end{enumerate}
\end{coro}
\begin{proof}
By Theorem \ref{thm:qsymsubalgquotRBA}, $\sha^{+}(x)$ is isomorphic to $\Wcqsym$ so that it is a Hopf algebra.
Notice that $\sha_{\bfk}(x)$ is the tensor product algebra $\bfk[x]\otimes \sha^{+}(x)$, where we write $f(x)\otimes \mathbf{1}_{\bfk}=f(x)$ for $f(x)\in \bfk[x]$,
so $\sha_{\bfk}(x)$ is a Hopf algebra.

By Theorem \ref{freerbamsp}, Parts \eqref{coro:pdiamo} and \eqref{coro:mup} hold. Part \eqref{coro:codeltaph} follows at once from Eqs.~\eqref{mqsydeltcop1}, \eqref{eq:Deltakxxm} and \eqref{gener7alabdcop},
Part \eqref{coro:epsilon} can be read from Eqs.~\eqref{mqsydeltcop2}, \eqref{gener7counit} and the formula $\epsilon_{\bfk[x]}(x^m)=\delta_{0,m}$,
and Part \eqref{coro:pantipodehp} follows immediately from Eqs.~ \eqref{antipodeqesym0}, \eqref{eq:smk[x]} and \eqref{anti96shax}. This completes the proof.
\end{proof}

\noindent
{\bf Acknowledgements.}
This work was partially supported by the National Natural Science Foundation (Grant No. 11371178 and 11501467) and Chongqing Research Program of Application
Foundation and Advanced Technology (No. cstc2014jcyjA00028). The authors thank William Y. Sit for stimulative discussions.
H.\ Yu thanks the hospitality and stimulating environment provided by New York University.

\end{document}